\documentclass{article}
 \usepackage{float}
\usepackage{verbatim}
\usepackage{ragged2e} 
\usepackage{amsmath}
\usepackage{float}
\usepackage{caption}
\usepackage{subcaption}
\usepackage{amssymb}
\usepackage{mathrsfs}
\usepackage{euscript}
\usepackage{amsthm}
\usepackage{mathtools}
\numberwithin{equation}{section}
\usepackage{bigints}

\DeclareMathOperator*{\divi}{div}
\usepackage{scalerel,stackengine}
\usepackage{nameref}
\usepackage{varioref}
\usepackage[breaklinks,colorlinks]{hyperref}
\usepackage[colorlinks]{hyperref}
\hypersetup{colorlinks=true}
\hypersetup{    colorlinks=true,  linkcolor=blue,  filecolor=blue,        urlcolor=blue,  anchorcolor=black,  citecolor={blue},}
\usepackage{cleveref}

\usepackage{lipsum}
\providecommand\abstractname{Abstract}
\def\abstract{}

\renewenvironment{abstract}{%
  \centering\small
  \textbf\abstractname
  \list{}{\leftmargin0.2cm \rightmargin\leftmargin}
  \item\relax
}{%
  \endlist \par\bigskip
}

\advance\oddsidemargin-\textwidth
\evensidemargin=\oddsidemargin

\theoremstyle{plain}
\newtheorem{theorem}{Theorem}[section]
\newtheorem{lemma}[theorem]{Lemma}

\newtheorem{coro}[theorem]{Corollary}

\theoremstyle{definition}

\theoremstyle{remark}

\newcommand{\R}{\mathbb{R}}

\def\Xint#1{\mathchoice
	{\XXint\displaystyle\textstyle{#1}}%
	{\XXint\textstyle\scriptstyle{#1}}%
	{\XXint\scriptstyle\scriptscriptstyle{#1}}%
	{\XXint\scriptscriptstyle\scriptscriptstyle{#1}}%
	\!\int}
\def\XXint#1#2#3{{\setbox0=\hbox{$#1{#2#3}{\int}$}
		\vcenter{\hbox{$#2#3$}}\kern-.5\wd0}}

\def\dashint{\Xint-}

\begin{document}

\title{	 \vspace{-2cm} \textbf{Multiplicity and Bifurcation
		Results for a Class of Quasilinear Elliptic Problems with Quadratic Growth on the Gradient} }

\markboth{Multiplicity and Bifurcation Results for a Class of Quasilinear Elliptic Problems 
}
         {Multiplicity and Bifurcation Results for a Class of Quasilinear Elliptic Problems 
}

\def\authorOne{\authorfont{Fiorella Rendón}}
\def\authorTwo{\authorfont{Mayra Soares}}
\def\institutionOne{\subauthorfont{Department of Mathematics, PUC-RIO}}
\def\institutionTwo{\subauthorfont{Department of Mathematics, University of Brasilia}}

\author{Fiorella Rend\'on and Mayra Soares}

\maketitle

\setcounter{tocdepth}{3}

%
%
%
%
\maketitle
%
\bgroup
 \vspace{-0.5cm}
\hrule
\egroup

%
\bgroup
\egroup

%
\bigskip 
\begin{abstract}
We investigate the existence, non-existence, and multiplicity of solutions to the following class of quasilinear elliptic equations
	\begin{align*}\tag{$P_\lambda$}
		-\mathrm{div}(A(x)Du)=c_\lambda(x)u+( M(x)Du,Du)+h(x),\qquad
		u\in H_0^1(\Omega)\cap L^\infty(\Omega),
	\end{align*}
	where $\Omega\subset\mathbb{R}^n$, $n\geq 3$, is a bounded domain with a low-regularity boundary 
 $\partial\Omega$. 
 
 The coefficients $c, h \in L^p(\Omega)$ for some $p > n$, with $c^\pm \geq 0$ and $c_\lambda(x) := \lambda c^+(x) - c^-(x)$ for a real parameter $\lambda$. The matrix $A(x)$ is uniformly positive definite and bounded, while $M(x)$ is positive definite and bounded.

Under suitable assumptions, we characterize the solution continuum of $(P_\lambda)$, including its bifurcation points. We establish existence and uniqueness results in the coercive case ($\lambda \leq 0$) and prove multiplicity results in the non-coercive case ($\lambda > 0$).

	\bigskip

	
\end{abstract}

\bgroup
\hrule
\egroup
\section{Introduction}

\quad \ We consider the following class of boundary value problems
\begin{align*}\tag{$P_\lambda$}\label{$P_lambda$}
	\left\{
	\begin{array}{rl}
		-\divi(A(x)Du)&=c_\lambda(x)u+( M(x)Du,Du)+h(x)\\
		u&\in H_0^1(\Omega)\cap L^\infty(\Omega)
	\end{array}
	\right.
\end{align*}
where $\Omega\subset\mathbb{R}^n$, for $n\geq 3$, is a $C^{1,\mathcal{D}ini}$ bounded domain and $c,h \in L^p(\Omega)$ for some $p>n$, with $c^+$ and $c^-$  non-negative functions such that $c_\lambda(x):=\lambda c^+(x)-c^-(x)$ for a parameter $\lambda\in\mathbb{R}$. Furthermore, $A(x) \in C^{0,\mathcal{D}ini}$ is a uniformly positive bounded measurable matrix, i.e. $$\vartheta I_n \leq (a_{ij}(x))\leq \vartheta^{-1}I_n,$$ for a positive constant  $\vartheta$, and $I_n$ is the identity matrix; and that $M(x)$ is a positive matrix such that
\begin{align}\label{5.2}
	0<\mu_1I_n \leq M(x) \leq \mu_2 I_n \quad \mbox{ in } \Omega,
\end{align}
for some positive constants $\mu_1$ and $\mu_2$.

We say that $u$ is a weak Sobolev (super, sub) solution to \eqref{$P_lambda$}, if $u \in H^1(\Omega)$ is such that, for any test function  $ \varphi\in H_0^1(\Omega)$, we have
\begin{align*}
	\int_\Omega A(x)DuD \varphi \, (\geq, \leq) = \int_\Omega c_\lambda(x) u \varphi + \int_\Omega\varphi(M(x)Du,Du)+\int_\Omega h(x)\varphi.
\end{align*}
 Depending on the parameter, $\lambda\in\mathbb{R}$ we study the existence and multiplicity of
solutions to \eqref{$P_lambda$}.

The class of problems \eqref{$P_lambda$} is challenging and more delicate due to the quadratic dependence in the gradient, which gives to the gradient term the same order as the Laplacian, with respect to dilations. We refer to \cite{ CJ, MR4030257, multiplicidade} for a comprehensive review of the extensive literature on this topic.

The study of the coercive case ($c \leq 0$) was initiated by Boccardo, Murat, and Puel in the 1980s, with the uniqueness of solutions established in \cite{ACJTuni}. Notably, the available results for the coercive case primarily address the non-divergence form of the problem. In contrast, the non-coercive case remained largely unexplored until recently. A pioneering example of the non-coercive, divergence-form case was studied by Jeanjean and Sirakov \cite{Jeanjean-2013}, 
 where they studied a problem directly connected to $(P_\lambda)$, which serves as a special case of our general framework.

For problems with a varying positive coefficient ($c_\lambda \gneqq 0$), the authors of \cite{ACJT, CJ} employed topological methods, exploiting the fact that the problem lies between two constant-coefficient problems of the form $\mu_1|Du|^2$ and $\mu_2|Du|^2$ ($\mu_1 > 0$). This allowed them to derive the necessary a priori bounds using a modified Brezis-Turner approach. In contrast, \cite{MR4030257} addressed the case where $\mu$ changes sign, making the comparison technique mentioned above inapplicable. Instead, the authors introduced a novel method to control solutions on $\Omega_{c^+}$, distinguishing between interior ($\Omega \cap \overline{\Omega}_{c^+}$) and boundary ($\partial\Omega \cap \overline{\Omega}_{c^+}$) cases. For the former, they applied the Interior Weak Harnack Inequality (IWHI), while for the latter, they used the Boundary Weak Harnack Inequality (BWHI). These techniques were originally applied to the case of the Laplacian operator.  In this work, we adapt and extend these techniques to the divergence-form case.

Since $c_\lambda$ changes sign, global sign conditions are unavailable, and the approaches used in \cite{ACJT, CJ} for obtaining a priori bounds cannot be applied. Instead, we adopt the framework of \cite{MR4030257} and assume the additional hypothesis
\begin{align}\label{A}\tag{$A_c^+$}
	\left\{
	\begin{array}{rl}
		\Omega_{c^+}:=\mathrm{supp}(c^+), \	&|\Omega_{c^+}|>0 \mbox{ and there exists}\\
		\varepsilon>0 \mbox{ such that } c^-=0& \mbox{in } \{x\in\Omega:d(x,\Omega_{c^+})<\varepsilon\}. \end{array}
	\right.
\end{align}

This assumption captures the ``hard" non-coercive scenario, where the zero-order coefficient is non-negative, and uniqueness of solutions is expected to fail. Theorem \ref{teo5.2}(iii) illustrates this behavior, demonstrating the multiplicity of solutions. For the definition of $\mathrm{supp}(f)$ when $f \in L^p(\Omega)$, we refer to \cite[Proposition 4.17]{MR2759829}.

Prior works such as \cite{JR, MR4098038} also treated the non-coercive case but simplified the analysis by assuming a gradient term of the form $\mu|Du|^2$ with $\mu$ constant. This permits a change of variables that eliminates the gradient term, allowing the use of variational methods.

Let $\gamma_1 > 0$ denote the principal eigenvalue of the linear problem associated with \eqref{$P_lambda$}, namely,
\begin{align}\label{eig1}\tag{$P^a_{\gamma_1}$}
	\left\{
	\begin{array}{rll}
		-\divi(A(x)D\varphi_1)&=c_{\gamma_1}(x) \varphi_1 &\mbox{ in } \Omega\\
		\varphi_1&> 0&\mbox{ in } \Omega\\
		\varphi_1&= 0&\mbox{ on } \partial\Omega.
	\end{array}
	\right.
\end{align}
In view of \cite[Theorems 8.37-8.38]{GT}, problem \eqref{eig1} admits a solution $\varphi_1$, the principal eigenfunction associated with $\gamma_1$. We note that for $\lambda = \gamma_1$ and $h(x) \gneqq 0$, problem \eqref{$P_lambda$} has no solution $u$ with $c^+(x)u \gneqq 0$, nor non-negative solutions when $\lambda \geq \gamma_1$ (see \cite[Lemma 6.1]{ACJT} for details).

In what follows, a continuum means a closed and connected set, and the above assumptions on the coefficients of the equation are assumed to hold.
More precisely, defining
\[
\Sigma:=\{(\lambda,u)\in \mathbb{R}\times C(\overline{\Omega}): u \mbox{ solves } \eqref{$P_lambda$}\},
\]
our primary contribution is a complete description of the solution set $\Sigma$ of \eqref{$P_lambda$}. Inspired by \cite{MR4030257}, in the next two theorems we prove the existence of a continuum of solutions under the assumption that the coercive problem ($P_0$) admits a solution. 
We point out that, problem ($P_0$) is the particular case of problem \eqref{$P_lambda$}, when $\lambda = 0$ or $c^+ \equiv 0$.
This limiting case is independent of $\lambda$ and admits a unique solution, as established in 
\cite[Theorem 1]{Jeanjean-2013}, where the authors derived sufficient (additional) smallness 
conditions for the existence of solutions to $(P_0)$ for general divergence-form operators. 
For related results concerning the Laplacian and $p$-Laplacian cases, see also \cite{ACJT,CF}.

\begin{theorem}\label{teo5.2}
	Suppose that $(P_0)$ has a solution $u_0$ with $c^+(x)u_0\gneqq 0$. Then
	\begin{itemize}
		\item[(i)] For all $\lambda\leq 0$,  (\ref{$P_lambda$}) has a unique solution $u_\lambda$, which satisfies $u_0-\|u_0\|_\infty\leq u_\lambda\leq u_0$;
		\item[(ii)] There exists a continuum $\mathcal{C}\subset\Sigma$ such that the projection of $\mathcal{C}$ on the $\lambda$-axis is an unbounded interval $(-\infty,\overline{\lambda}]$ for some $\overline{\lambda}\in(0,+\infty)$ and $\mathcal{C}$ bifurcates from infinity to the right of the axis $\lambda=0$;
\item [(iii)]There exists $\lambda_0\in (0,\overline{\lambda}]$ such that for all $\lambda\in(0, \lambda_0)$, the problem
   (\ref{$P_lambda$}) has at least two nontrivial solutions $u_{\lambda,i}$, for $i=1,2$, with 	$u_0 \leq	u_{\lambda,1}\ll  u_{\lambda,2}$  and hence $\displaystyle\min_{\overline{\Omega}} u_{\lambda,2}>0.$
In addition, if $0<\lambda_1<\lambda_2$, then $ u_0\leq u_{\lambda_1,1}\leq u_{\lambda_2,1}$ and the
problem  $(P_{\overline{\lambda}})$ has a unique solution $u_{\overline{\lambda}} \geq u_0$.
\end{itemize}	
	
\end{theorem}

We remark that from Theorem \ref{teo5.2} we know that for $\lambda > \overline{\lambda}$, (\ref{$P_lambda$}) has no non-negative solutions, however, it does not
exclude the possibility of having negative or sign-changing solutions.    

\begin{theorem}\label{teo5.3}
	Suppose that $(P_0)$ has a solution $u_0 \leq 0$ with $c^+(x)u_0\lneqq 0$. Then
	\begin{itemize}
		\item[(i)] For  $\lambda\leq 0$,  (\ref{$P_lambda$}) has a unique non-positive solution $u_\lambda$ and this solution satisfies $ u_0+\|u_0\|_\infty \geq u_\lambda\geq u_0$;
		\item[(ii)] There exists a continuum $\mathcal{C}\subset\Sigma$ such that its non-negative projection $\mathcal{C}^+$ on the $\lambda$-axis is $[0,+\infty)$;
		\item [(iii)] For $\lambda>0$, every non-positive solution to (\ref{$P_lambda$}) satisfies $u_\lambda\ll u_0$. Furthermore, (\ref{$P_lambda$}) has at least two nontrivial solutions $u_{\lambda,i}$, for $i=1,2$, with
	$
			u_{\lambda,1}\ll u_0\leq u_{\lambda,2},$  and $\displaystyle\max_{\overline{\Omega}} u_{\lambda,2}>0.$
		Moreover, if $0<\lambda_1<\lambda_2$ we have
		$ u_{\lambda_2,1}\leq u_{\lambda_1,1}\leq u_0.$
	\end{itemize}
\end{theorem}

	\begin{figure}[h!]
	\centering
	\begin{subfigure}{.49\textwidth}
		\centering
	\includegraphics[width=0.8\linewidth]{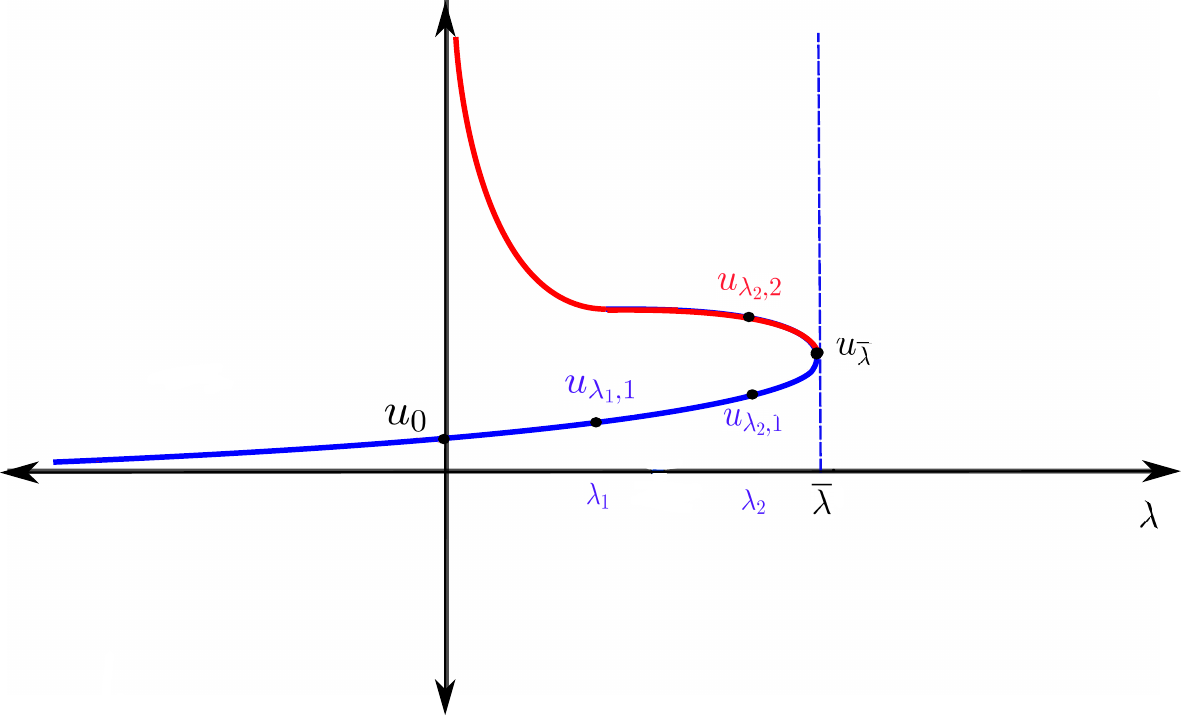}
	\caption*{Illustration of Theorem 1.1}
	\label{fig: Illustration of Theorem 1.1}
	\end{subfigure}
	\begin{subfigure}{.49\textwidth}
		\centering
	\includegraphics[width=0.7\linewidth]{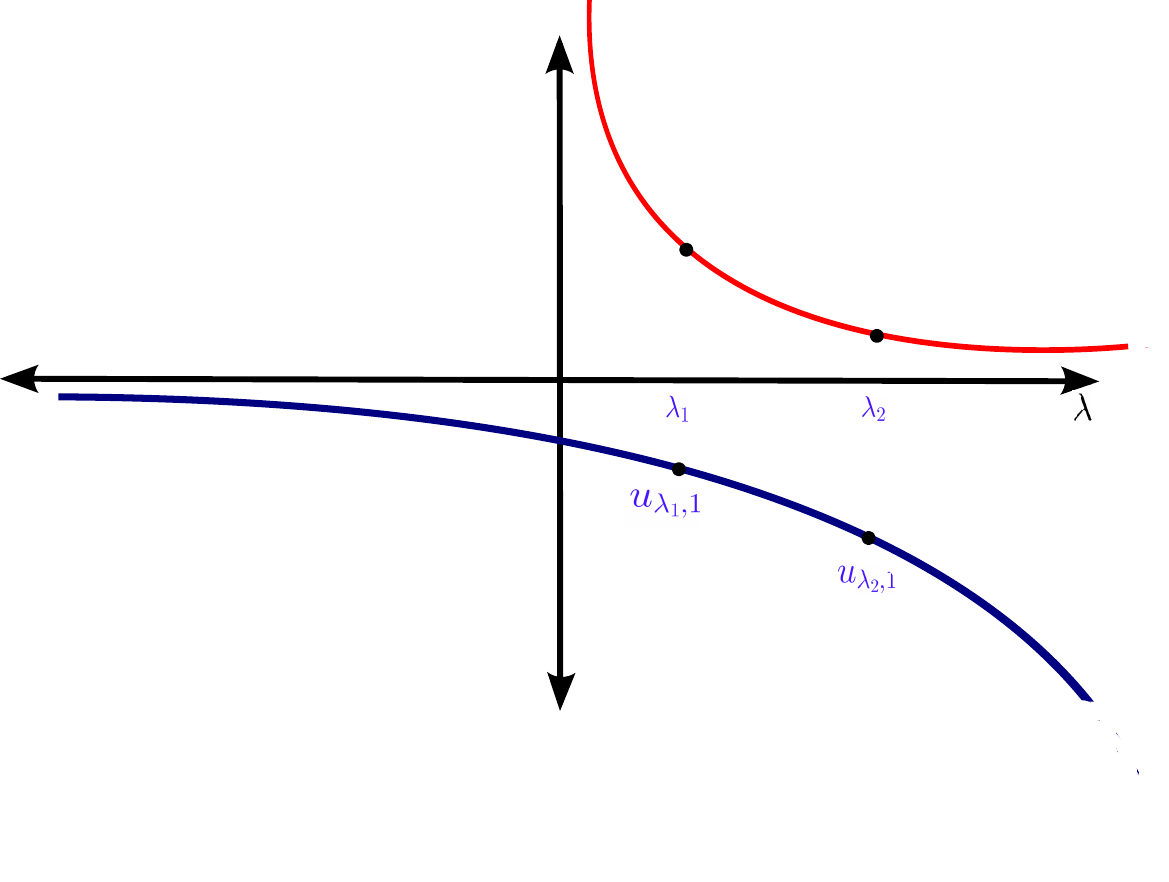}
	\caption*{Illustration of Theorem 1.2}
	\label{fig: Illustration of Theorem 1.2}
	\end{subfigure}
\end{figure}

In order to prove Theorem \ref{teo5.3} we consider an auxiliary problem \eqref{P3}, whose solutions are supersolutions to (\ref{$P_lambda$}).
Consider the problem
	\begin{align*}\tag{$P_{\lambda,{k}}$}
		\left\{
		\begin{array}{rll}
			-\divi(A(x)Du)=c_\lambda(x)u+(M(x)Du,Du)+h(x)+k\widetilde{c}(x) \mbox{ in } \Omega\\
			\qquad\qquad\qquad u= 0\qquad\qquad\qquad\qquad\qquad\qquad\qquad\qquad \mbox{ on } \partial\Omega
		\end{array}
		\right.
	\end{align*}
	for $k, \lambda$ and $\widetilde{c}$   to be defined posteriorly. 
Then, from Theorem \ref{teo5.2} and Lemma \ref{c2} we are able to deduce the following corollary, which concerns to the case $h\lneqq 0$. In this result, we can see the achievement of the two above theorems simultaneously.

\begin{coro}\label{coro} Assume that $h\lneqq 0$. For all $\widetilde{\lambda}>\gamma_1$, where
	$\gamma_1>0$  is the first eigenvalue \eqref{eig1}, there exists $\widetilde{k}>0$ such that, for all $k \in (0,\widetilde{k}]$,
	\begin{enumerate}
		\item[(i)] there exists $\lambda_1\in (0,\gamma_1)$ such that
		\begin{enumerate}
			\item for all $\lambda\in (0,\lambda_1)$,  \eqref{P3}
			has at least two positive solutions;
			\item for $\lambda=\lambda_1$,  \eqref{P3} has exactly one positive solution;
			\item for $\lambda> \lambda_1$, \eqref{P3} has no non-negative solution;
		\end{enumerate}
		\item[(ii)] for $\lambda=\gamma_1$ \eqref{P3} has no solution;
		\item[(iii)] there exists $\lambda_2\in (\gamma_1,\widetilde{\lambda}]$ such that
		\begin{enumerate}
			\item for $\lambda>\lambda_2$,  \eqref{P3} has at least two solutions with $u_{\lambda,1}\ll 0$ and $\min u_{\lambda,2}<0$;
			\item for $\lambda=\lambda_2$,  \eqref{P3} has a unique non-positive solution;
			\item $\lambda<\lambda_2$,  \eqref{P3} has no non-positive solution.
		\end{enumerate}
	\end{enumerate}
\end{coro}

Note that  Theorems \ref{teo5.2} and \ref{teo5.3} require to the problem $(P_0)$ to have a solution, thus we are in the situation  that a branch of solutions emanates from $(0, u_0)$. Our next results consider the alternative situation when problem $(P_0)$ does not have a solution, but there exists a non-positive supersolution to problem (\ref{$P_lambda$}) for some $\lambda_0 > 0$.
\begin{theorem}\label{solucoesnegativas}
	Assume that $(P_0)$ does not have a solution $u_0 \leq 0$ and that there exist $\lambda_0>0$ and $\beta_0\leq 0$ a supersolution to $(P_{\lambda_0})$.
	Then, there exists $0 <\underline{\lambda} \leq \lambda_0$ such that
	\begin{itemize}
		\item[(i)] for every $\lambda\in (\underline{\lambda},\infty)$,  (\ref{$P_lambda$}) has at least two solutions with $u_{\lambda,1}\leq 0$ and $u_{\lambda,1}\leq u_{\lambda,2}$.
		Moreover, if $\lambda_1 < \lambda_2$, we have $u_{\lambda_1,1}\gg u_{\lambda_2,1}$;
		\item[(ii)]  $(P_{\underline{\lambda}})$ has a unique solution $u_{\underline{\lambda}} \leq 0$;
		\item[(iii)] for $\lambda < \underline{\lambda}$,  $(P_{\lambda})$  has no solution $u\leq 0$.
	\end{itemize}
	Furthermore, for every $\lambda<0$ problem \eqref{$P_lambda$} has at most one non-positive solution $u_\lambda$, there exists an unbounded continuum $\mathcal{C} \subset \Sigma$ and $\lambda=0$ is a bifurcation point from infinity.
\end{theorem}
\begin{figure}[h!]	
	\centering
	\includegraphics[width=0.5\linewidth]{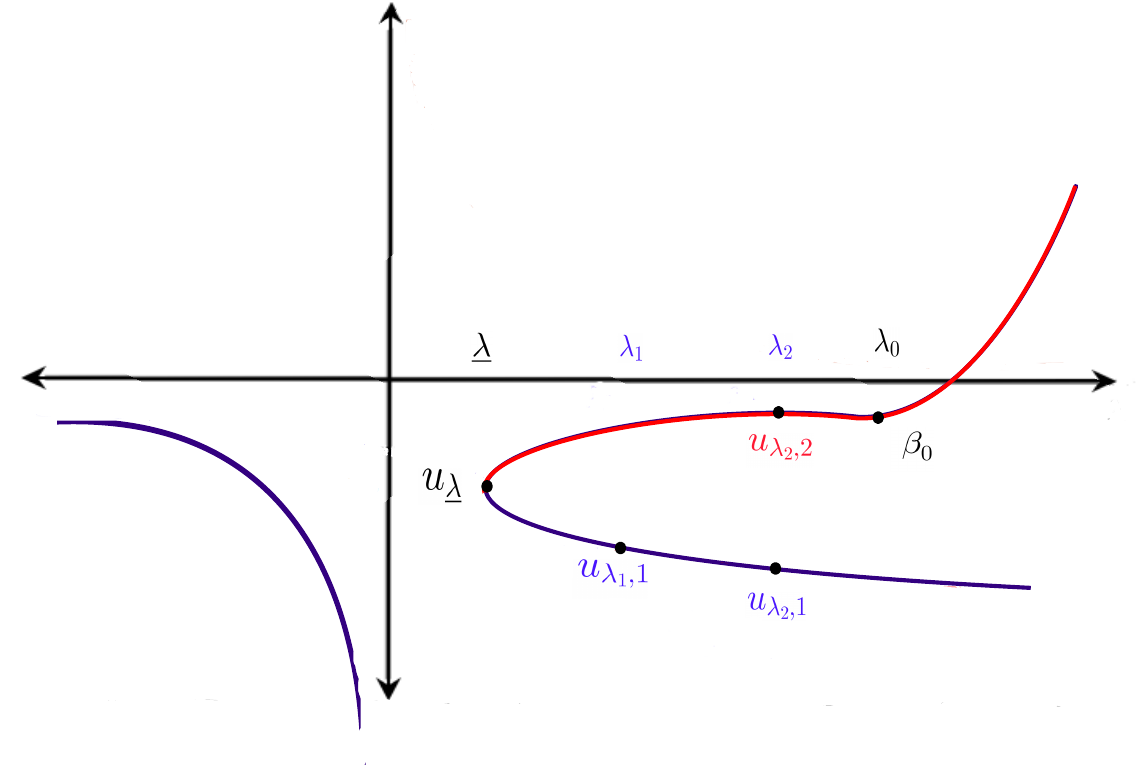}
	\caption*{\\ Illustration of Theorem 1.3}
	\label{fig: Illustration of Theorem 1.3}
\end{figure}
In the particular, but important, case $h(x)\equiv 0 $, we have the following result. Further considerations about the cases when  $h(x)$ has a sign are given in Section 4.
\begin{theorem}\label{hequiv0}
	Under assumption \eqref{A} with $h(x)\equiv 0$ and $\gamma_1>0$ 	is the first eigenvalue of \eqref{eig1}. Then
	\begin{itemize}
		\item[(i)] for all $\lambda\in (0,\gamma_1)$, the problem
		\begin{align}\label{h=0}\tag{$P_{h\equiv 0}$}\begin{cases}
			-\divi(A(x)Du)&=c_\lambda(x)u+(M(x)Du,Du) \qquad \mbox{ in }\Omega\\
		 \qquad\qquad\qquad u&= 0\qquad\qquad\quad \qquad\qquad\qquad \mbox{ on } \partial\Omega
            \end{cases}
		\end{align}
		has at least two solutions $u_{\lambda,1}\equiv 0$ and $u_{\lambda,2}\gneqq 0$; 
		\item[(ii)] for $\lambda=\gamma_1$, (\ref{h=0})  has only the trivial solution;
		\item[(iii)] for $\lambda >\gamma_1$,  (\ref{h=0})  has at least two solutions $u_{\lambda,1}\equiv 0$ and $u_{\lambda,2}\leq 0$;
		\item[(iv)] for all $\lambda\leq 0$, (\ref{h=0}) has a unique solution $u_\lambda\equiv 0$;
		\item[(v)] There exists a continuum $\mathcal{C}\subset\Sigma$ that bifurcates from infinity to the right of the axis $\lambda=0$ and whose projection on the $\lambda$-axis is an unbounded interval $(0,+\infty)$.
	\end{itemize}
\end{theorem}
	\begin{figure}
	\centering
	\begin{subfigure}{.49\textwidth}
		\centering
	\includegraphics[width=1.0\linewidth]{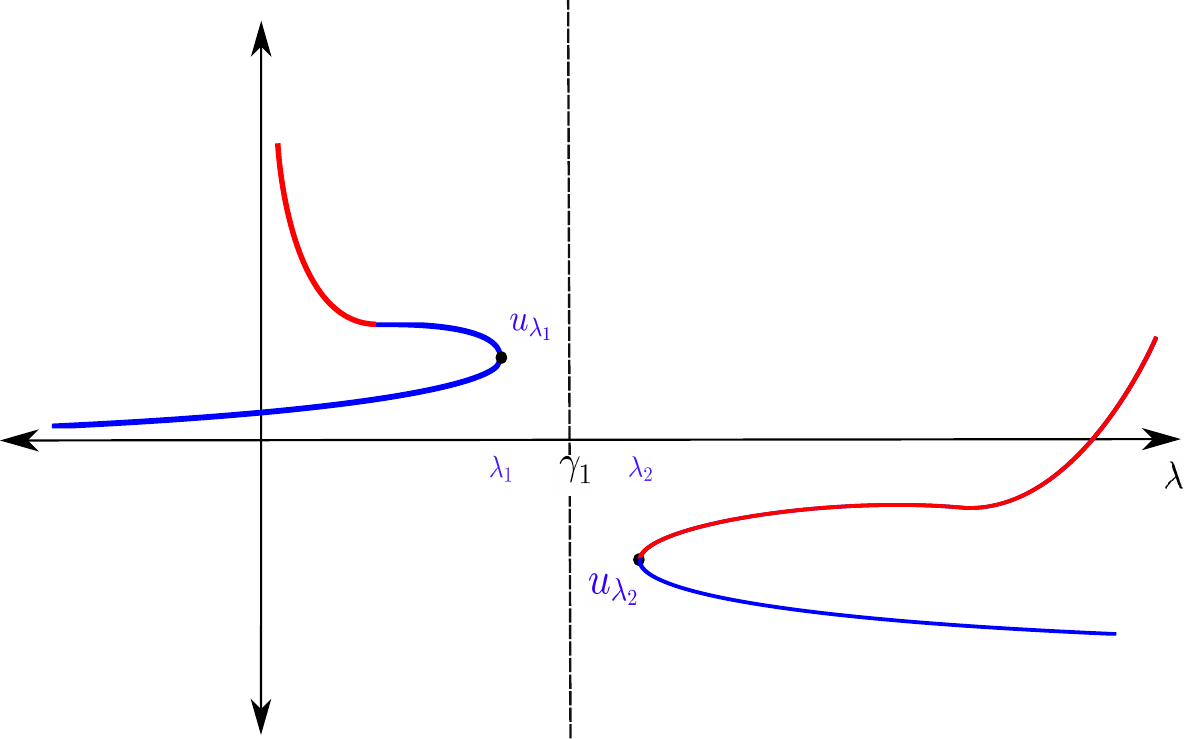}
	\caption*{\\ Illustration of Corollary 1.4}
	\label{fig:corollary}
	\end{subfigure}
\begin{subfigure}{.49\textwidth}
	\centering
	\includegraphics[width=1.0\linewidth]{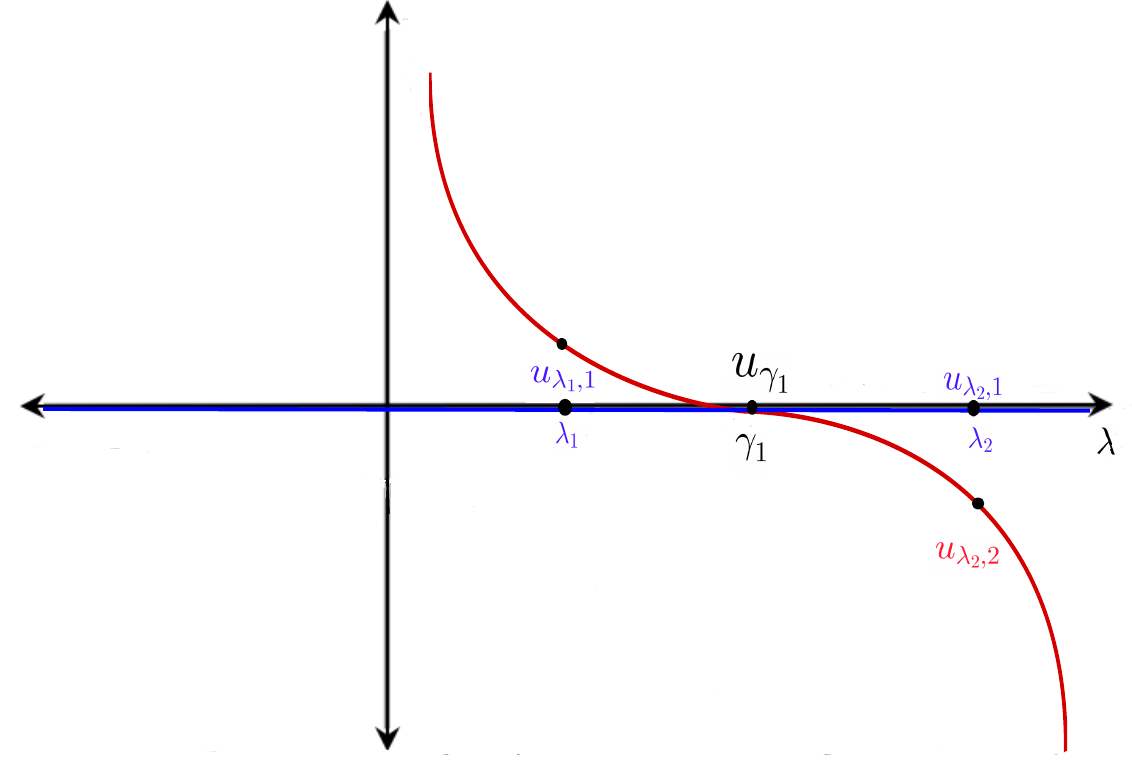}
	\caption*{Illustration of Theorem 1.5}
	\label{fig: Illustration of Theorem 1.5}
	\end{subfigure}
\end{figure}
The proof of our results relies crucially on the Boundary Weak Harnack Inequality for uniformly elliptic equations in divergence form, as established in \cite{RSS}. Specifically, Lemma 3.2 demonstrates that controlling the behavior of solutions on $\Omega_{c^+}$ is sufficient for our analysis. For this purpose, given $\bar{x} \in \Omega_{c^+}$, we perform a local analysis:
\begin{itemize}
    \item In a ball $B(\bar{x}, r) \subset \Omega$ if $\bar{x} \in \Omega_{c^+} \cap \Omega$ (interior case),
    \item In a semi-ball $B^+(\bar{x}, r)$ if $\bar{x} \in \Omega_{c^+} \cap \partial\Omega$ (boundary case).
\end{itemize}

The results presented in this work not only provide a novel approach to boundary analysis but also extend the existing theoretical framework.
It is also important to mention that, the existence statements in Theorems 1.1--1.5 are established through an auxiliary fixed point problem constructed using degree theory. This approach is based on the work  \cite{MR4030257}, where the authors developed a fixed point framework for the $p$-Laplacian with $c^+(x)u_0 \gneqq 0$. Our Theorem 1.1 generalizes these arguments to our setting. On the other hand,
for Theorem 1.2, where $c^+(x)u_0 \lneqq 0$, we introduce a new fixed point problem tailored to this case. Furthermore, we show that it is possible to obtain a more precise  characterization of the solution set $\Sigma$, particularly when the sign of $u_0$ is known.

This paper is organized as follows. Section 2 presents auxiliary results, which are fundamental for the construction of our arguments. In Section 3 we derive a priori bounds for the solutions of problem \eqref{$P_lambda$}. Finally, in Section 4 we prove our main results. 
\section{Auxiliary Results}\label{preresults}

\quad \ The Strong Maximum Principle is extremely important in our approach. As stated below, it guarantees that a non-negative supersolution to an elliptic equation in a domain cannot vanish inside the domain, unless it vanishes identically.

\begin{theorem}[{\bf{Strong Maximum Principle - SMP}}]\label{SMP}
	Let  $\Omega\subset \mathbb{R}^n$ be a domain. If $u$ satisfies
	\begin{align*}
		\left\{
		\begin{array}{rll}
			-\divi(A(x)Du)-\mu_1 |Du|^2-c_\lambda(x)u-h(x)&\geq 0 &\mbox{ in }\Omega\\
			u &\geq 0 &\mbox{ in }\Omega\\
		\end{array}
		\right.
	\end{align*}
	then either $u>0$ in $\Omega$ or $u\equiv 0$ in $\Omega$.
\end{theorem}

We observe that the SMP is an immediate consequence of the Interior Weak Harnack Inequality (IWHI). For more details and the proof of this theorem, we refer to \cite[Theorem 3.5, Theorem 8.18]{GT}.
On the other hand, the next theorem is a generalization of the IWHI up to the boundary. 
In order to enunciate it, we need some definitions.

	A function $\sigma:[0,1]\to \mathbb{R}_+$ is a {\bf Dini function} (write $\sigma \in \mathcal{D}$)  if

$(i) \ \sigma(0)=0< \sigma(t)/2 \leq \sigma(s)\leq \sigma(t)$ for $0< t/2 
\leq s\leq t$;

$(ii) \ \sigma(\tau)/\tau$ is non-increasing and $\int_0^s \frac{\sigma(\tau)}{\tau} d\tau < +\infty$.

We say $\psi:\Omega \to \mathbb{R}$ is a { \bf Dini continuous} function in $\Omega$ and write ${\psi \in \mathcal{C}^{0,\mathcal{D}ini}(\overline\Omega)}$ if there exists some $\sigma_\psi \in \mathcal{D}$ such that 
\[	|\psi(x)-\psi(y)|\leq \sigma_\psi(|x-y|) \quad \text{for \ all} \quad x, y \in \overline{\Omega}.
\]
Then,  $\Omega$ is a  $\mathcal{C}^{1, \mathcal{D}ini}$ domain if, locally, 
{  $\partial \Omega$ can be seen 
	as the graph of a $\mathcal{C}^1$-function}, whose derivatives are of class   $\mathcal{C}^{0,\mathcal{D}ini}$.

Setting $B_R^+ = B_R \cap \Omega$, we say a function $\psi$ has { \bf Dini mean oscillation on $\Omega$} and write { \bf ${\psi \in \mathcal{C}^{0,m\mathcal{D}ini}(\overline{\Omega})}$}  if there exists  $\sigma_{m,\psi} \in \mathcal{D}$ such that 
\[
\dashint_{B^+_R(x)} |\psi(y)-\dashint_{B^+_R(x)}\psi(z)dz|dy\leq \sigma_{m,\psi}(R) \ \mbox{ for \ every } \ R>0, \ x \in \overline{\Omega}.
\]				

If $\psi \in \mathcal{C}^{0,m\mathcal{D}ini}(\overline\Omega)$ and $\Omega$ is a $\mathcal{C}^1$ domain, then $\psi$ is uniformly continuous in $\Omega$, with a modulus
of continuity $\omega_\psi(r)$ dominated by $\int_0^r \frac{\sigma_m(\tau)}{\tau} d\tau.$ However, we remark that $\sigma_{m,\psi}(r) \leq \sigma_{\psi}(r)$, so the Dini mean oscillation is a weaker hypothesis
than Dini continuity.

Now, we enunciate the Boundary Weak Harnack Inequality. Such a result is the core of our arguments in order to describe the solutions obtained to the class of problems \eqref{$P_lambda$}. Its proof can be found in \cite[Theorem 4.7]{fio}, see also \cite[Theorem 1.1]{RSS} for a more general version. 
\begin{theorem}[{\bf{Boundary Weak Harnack Inequality, BWHI}}]\label{Dini}\label{bwhi} We assume that
	\begin{itemize}
		\item[(H1)] $A(x)\in L^\infty(\Omega)$, $\lambda I\leq A(x)\leq \Lambda I$ for some $0< \lambda\leq \Lambda$;
		\item[(H2)] $b \in L^q_{loc}(\Omega)$ for some $q>n$, $c,f \in L^p_{loc}(\Omega)$ for some $p>n/2$;
	\end{itemize}
	We set $d(x) = dist(x,\partial\Omega)$, 
	$\Omega_{d_0} =\{x \in\Omega
	: d(x) < d_0\}$, and assume that
	\begin{itemize}
		\item[(H3)]  $\Omega$ is $C^{1,\mathcal{D}ini}$ domain, the coefficients of $A$ have Dini mean oscillation in 
		$\Omega_{d_0}$, for some  $d_0 > 0$,
 $u \geq 0 $ in $\Omega$ and satisfies
	\[ -\divi(A(x)Du)+b(x)Du+c(x)u \geq f \quad \mbox{ in } \Omega.\]
\end{itemize}
	 Then there exist $\varepsilon>0$ depending on $n, \lambda, \Lambda,p, q, \sigma_{m,A} $ and $C>0$  depending on $n, \lambda, \Lambda, p,q, \sigma_{m,A}$, $\|b\|_{L^q(\Omega)},$ $\|c\|_{L^p(\Omega)}$ and $\partial \Omega$, such that
	\begin{align}\label{IWHIloc1}
		\left(\int_{\Omega} \left(\frac{u}{d}\right)^\varepsilon\right)^{1/\varepsilon}\le C\left(\inf_{B_R^+} \frac{u}{d} + \|f\|_{L^p(\Omega)}\right).
	\end{align}	
\end{theorem}

Before stating the next auxiliary result, we denote by $C(\overline{\Omega})$ the real Banach space of continuous functions defined over $\overline{\Omega}$ and let ${T:\mathbb{R}\times C(\overline{\Omega})\rightarrow C(\overline{\Omega})}$ be a completely continuous map, i.e. it is continuous and maps bounded sets to relatively compact sets. For the purposes of this paper, we consider the problem 
\begin{equation}\label{Q}
	u \in C(\overline{\Omega});\quad \Phi(\lambda,u):=u-T(\lambda,u)=0,
\end{equation}
of finding the zeroes of $\Phi(\lambda,u):=u-T(\lambda,u)$, for each fixed $\lambda\in\mathbb{R}$. Let $\lambda_0\in\mathbb{R}$ be arbitrary but fixed and assume that $u_{\lambda_0}$ is an isolated solution to $\Phi(\lambda_0,u)$, then the degree $\deg(\Phi(\lambda_0,.),B(u_{\lambda_0},r),0)$ is well defined and it is constant for $ r>0$ small enough. Thus, it is possible to define the index
\[
	i(\Phi(\lambda_0,.),u_{\lambda_0}):=\lim_{r\rightarrow 0}\deg(\Phi(\lambda_0,.),B(u_{\lambda_0},r),0). 
\]
Now we are able to enunciate the following theorem, which was proved in \cite[Theorem 2.2]{ACJT}.
\begin{theorem}\label{continuum}
	If (\ref{Q}) has a unique solution $u_{\lambda_0}$, and $i(\Phi(\lambda_0,.),u_{\lambda_0})\ne 0$ then $\Sigma$ possesses two unbounded components $\mathcal{C}^+$ and $\mathcal{C}^-$ in $[\lambda_0,+\infty]\times C(\overline{\Omega})$ and $[-\infty,\lambda_0]\times C(\overline{\Omega})$, respectively, which meet at $(\lambda_0,u_{\lambda_0})$.
\end{theorem}

A fundamental tool for establishing the existence and behavior of solutions, using the lower and upper solutions approach, is the following theorem. This result is analogous to \cite[Theorem 2.1]{CJ}, originally proved for the Laplacian operator. 
\begin{theorem}\label{existence}
Let  $\Omega$ is a bounded domain in $\R^n$ with boundary $\partial\Omega$ of class $C^{1, \mathcal{Dini}}$ and $f$ be an $L^p$-Carath\'{e}odory function with $p > n$. Assume that there exists a lower solution $\alpha$ and an upper solution $\beta$ of \eqref{$P_lambda$} such that $\alpha\leq\beta$.
Denote $\alpha := \max \{\alpha_i| 1\leq i\leq k \}$ where $\alpha_1,\cdots,\alpha_k$ are regular lower solutions to problem
\eqref{$P_lambda$} and $\beta := \min \{\beta_j| 1\leq j\leq l \}$ where $\beta_1,\cdots, \beta_l$ are regular upper solutions
of \eqref{$P_lambda$}. If there exists $K > 0$ and $h\in L^p(\Omega)$ such that for a.e. $x \in\Omega$, all $\xi\in\R^n$,
\[
|f(,u,\xi)|\leq h(x)+K|\xi|^2,
\]
then the problem \eqref{$P_lambda$} has at least one solution $u$ satisfying $\alpha \leq u\leq \beta$.
Moreover, problem \eqref{$P_lambda$} has a minimal solution $u_{min}$ and a maximal solution $u_{max}$ in the sense that, $u_{min}$ and $u_{max}$ are solution of \eqref{$P_lambda$} with $\alpha\leq u_{min}\leq u_{max}\leq \beta$ and every solution $u$ of \eqref{$P_lambda$} with $\alpha\leq u\leq \beta$ satisfies $u_{min}\leq u \leq u_{max}$.
If moreover $\alpha$ and $\beta$ are strict and satisfy $\alpha\ll \beta$, then there exists $R>0$ such that
\[
\deg(\mathcal{I-M,S})=1
\]
where 
\[
\mathcal{S}=\{ u\in C_0^1(\overline{\Omega})| \alpha\ll u \ll \beta , \|u\|_{C^1}<R\}.
\]
\end{theorem}
Notice that, in order to assign a degree to a pair of lower and upper solutions, it is necessary to consider a stronger class of functions.
We recall that a strict subsolution  to \eqref{$P_lambda$} is a subsolution $\alpha$ such that for every solution $u$ to  \eqref{$P_lambda$} satisfying $\alpha \leq u$ it follows that $\alpha\ll u$, namely, $\alpha(x)< u(x)$ for every $x\in \Omega$. As well as, a strict supersolution to \eqref{$P_lambda$} is a supersolution $\beta$ such that $u\ll \beta$, for every solution $u$ to \eqref{$P_lambda$} satisfying $u\leq\beta$.

In order to consider the situation where ($P_{\lambda_0}$) has a supersolution, we need the following formulation of the Anti-maximum Principle. This result was established in \cite[Theorem 1]{MR633824}, see also \cite[Proposition 2.6]{CJ} which presents a particular case of this theorem for the Laplacian operator.

\begin{lemma}\label{antimax}
	Let $\overline{c}, \overline{h}, \overline{d}\in L^p(\Omega)$ with $p>n$ and assume $\overline{h}\gneqq 0$. We denote by $\overline{\gamma}_1>0$ the first eigenvalue of
	\[
		-\divi(A(x)Du)+\overline{d}(x)u=\overline{c}_{\overline{\gamma}_1}(x)u, \mbox{ } u\in H^1_0(\Omega).
	\]
	Then there exists $\varepsilon_0>0$ such that, for all $\lambda\in(\overline{\gamma}_1,\overline{\gamma}_1+\varepsilon_0)$, the solution $v\in H^1_0(\Omega)$ of
	\begin{align*}
		-\divi(A(x)Dv)+\overline{d}(x)v=\overline{c}_{\lambda}(x)v+\overline{h}(x), \quad \mbox{	satisfies \ } v\ll 0.
	\end{align*}
\end{lemma}

\section{A priori Bounds}\label{aprioribound}
\quad \ This section is devoted to the derivation of some a priori bounds results for the solutions to \eqref{$P_lambda$}. Most of our results hold true under more general assumptions than \eqref{A}.
First, we obtain the following essential upper bound on the supersolutions to \eqref{$P_lambda$}, which guarantees that, for $\lambda>0$  in a bounded interval, any unbounded continuum of solutions to \eqref{$P_lambda$} can only bifurcate to the right of $\lambda=0$.
\begin{theorem}[{\bf{A priori Upper Bound}}]\label{6.3}
	Under the stated assumptions of problem \eqref{$P_lambda$}, including hypothesis \eqref{A}, for any $\Lambda_2>\Lambda_1>0$, there exists a constant $\widetilde{M}>0$ such that, for each $\lambda\in[\Lambda_1,\Lambda_2]$, any solution to \eqref{$P_lambda$} satisfies $\displaystyle\sup_{\Omega} u \leq\widetilde{M}$.
\end{theorem}

To prove Theorem \ref{6.3} in the general case, we show that an a priori bound for solutions of \eqref{$P_lambda$} depends only on controlling the solution over $\Omega_{c^+}$. By compactness, 
this reduces to analyzing the behavior near arbitrary points $\overline{x} \in \overline{\Omega}_{c^+}$.

\begin{lemma}\label{lbound}
Under the hypotheses to \eqref{$P_lambda$}, there exists a constant $M>0$ such that, for any $\lambda\in \mathbb{R}$, any solution $u$ to the problem \eqref{$P_lambda$} satisfies
	\begin{align*}
		-\sup_{\Omega_{c^+}}u^--M\leq u\leq\sup_{\Omega_{c^+}}u^++M.
	\end{align*}
\end{lemma}
\begin{proof}
	In case the problem \eqref{$P_lambda$} has no solutions for any $\lambda\in \mathbb{R}$, there is nothing to prove.
	Hence, we assume the existence of $\widetilde{\lambda}\in \mathbb{R}$ such that $( P_{\scriptscriptstyle{\widetilde{\lambda}}})$ has a solution $\widetilde{u}$.
	We shall prove the result for $M:=2\|\widetilde{u}\|_{\infty}$. Let $u$ be an arbitrary solution to \eqref{$P_lambda$}.
	Setting $\mathfrak{D}:=\Omega \setminus \overline{\Omega}_{c^+}$ and $v=u-\sup\limits_{\partial \mathfrak{D}} u^+$, we have
	\begin{eqnarray*}
		-\divi(A(x)Dv)&=&-c^-(x)v+(M(x)Dv,Dv)+h(x)-c^-(x)\sup\limits_{\partial \mathfrak{D}} u^+\\
		&\leq& -c^-(x)v+(M(x)Dv,Dv)+h(x)\mbox{ in } \mathfrak{D}.
	\end{eqnarray*}
	Since $v\leq 0$ on $\partial \mathfrak{D}$, it follows that $v$ is a subsolution to $(P_0)$.
		On the other hand, setting $\widetilde{v}=\widetilde{u}+\|\widetilde{u}\|_\infty$ we obtain
		\begin{eqnarray*}
			-\divi(A(x)D\widetilde{v})&=&-c^-(x)\widetilde{v}+(M(x)D\widetilde{v},D\widetilde{v})+h(x)+c^-(x)\|\widetilde{u}\|_\infty\\
			&\geq&-c^-(x)\widetilde{v}+(M(x)D\widetilde{v},D\widetilde{v})+h(x)\mbox{ in }\mathfrak{D},
		\end{eqnarray*}
		and thus, as $\widetilde{v}\geq 0$ on $\partial \mathfrak{D}$, it means that $\widetilde{v}$ is a supersolution to $(P_0)$. By standard regularity results, see for instance  \cite[Lemma 2.1]{ACJTuni}, we get $u, \widetilde{u} \in H^1(\Omega)\cap W^{1,n}_{loc}(\Omega)\cap C(\overline{\Omega})$ and hence, ${v, \widetilde{v} \in H_0^1(\mathfrak{D})\cap W^{1,n}_{loc}(\mathfrak{D})\cap C(\overline{\mathfrak{D}})}$ and the right-hand sides of the above inequalities are $L^n$ functions. Therefore, we are able to apply the 
		Comparison Principle \cite[Lemma 2.11]{fio} and conclude that $v\leq \widetilde{v}$ in $\mathfrak{D}$, namely,
		$
			u \leq  \widetilde{u}+\|\widetilde{u}\|_\infty +\displaystyle\sup_{\partial \mathfrak{D}} u^+ \mbox{ in } \mathfrak{D}
	$ and then, $u \leq  M +\displaystyle\sup_{\Omega_{c^+}} u^+ \mbox{ in } \Omega$.
	
		For the other inequality, we now define $v:=u+\sup\limits_{\partial\mathfrak{D}}u^-$ and  hence obtain that $v\geq 0$ on $\partial \mathfrak{D}$,
	and also that $v$ is a supersolution to $(P_0)$. Furthermore, defining $\widetilde{v}=\widetilde{u}-\|\widetilde{u}\|_{\infty}$, we have that $\widetilde{v}\leq 0$ on $\partial\mathfrak{D}$ 
	and also that $\widetilde{v}$ is a subsolution to $(P_0)$. As previously, we have that $v, \widetilde{v} \in H_0^1(\mathfrak{D})\cap W^{1,n}_{loc}(\mathfrak{D})\cap C(\overline{\mathfrak{D}})$, and applying again the Comparison Principle 
		we get $\widetilde{v}\leq v$ in $\mathfrak{D}$. Namely,
		$
	\geq \widetilde{u}-\|\widetilde{u}\|_{\infty}  u-\displaystyle\sup_{\partial\mathfrak{D}}u^- \mbox{ in } \mathfrak{D}.
	$
		Therefore, it yields $u\geq -\displaystyle\sup_{\Omega_{c^+}}u^--M \mbox{ in } \Omega$, ending the proof.
	\end{proof}
	
	Now, let $u\in H^1_0(\Omega)\cap L^\infty(\Omega)$ be a solution to \eqref{$P_lambda$}. We introduce the exponential change of variable
	\begin{align}
		w_i(x):=\frac{1}{\nu_i}(e^{\nu_i u(x)}-1) \quad\mbox{and}\quad g_i(x):=\frac{1}{\nu_i}\ln(1+\nu_i s)\mbox{, }i=1,2
	\end{align}
	where $\nu_1:=\mu_1\vartheta \ \mbox
	{ and }\ \nu_2:=\mu_2\vartheta^{-1}$, for $\mu_1$, $\mu_2$ given in \eqref{5.2} and $\vartheta$ given in the definition of the matrix $A(x)$.
	The following change of variables lemma follows straightway from an algebraic computation and it is going to be useful for proving our results.

	\begin{lemma}[{\bf Exponential Change}]\label{exponentialchange}
		Let $u$ be a weak solution to problem \[-\divi(A(x)Du)=f(x), \quad f\in L^p(\Omega).\]
		For $m>0$ we define
$ v:=\dfrac{e^{mu}-1}{m}$ and $w:=\dfrac{1-e^{-mu}}{m}.$
		Then $Dv=(1+mv)Du$, ${Dw=(1-mw)Du}$, and for each $\vartheta>0$ we have,
		\begin{align*}
			-\divi(A(x)Du)-\vartheta^{-1} m |Du|^2&\leq\frac{-\divi(A(x)Dv)}{1+mv}\leq -\divi(A(x)Du)-\vartheta m |Du|^2,\\
			-\divi(A(x)Du)+\vartheta m |Du|^2&\leq\frac{-\divi(A(x)Dw)}{1-mw}\leq -\divi(A(x)Du)-\vartheta^{-1} m |Du|^2,
		\end{align*}
		and $\{u=0\}=\{v=0\}$ and $\{u>0\}=\{v>0\}$.
		Therefore
		if $u$ is a weak supersolution to
		\begin{align}\label{expo}
			-\divi(A(x)Du)\geq \mu_1 |Du|^2+c_\lambda(x)u+h(x),
		\end{align}
		and for $m=\mu_1\vartheta$, $v$ is a weak supersolution to
		\begin{align*}
			-\divi(A(x)Dv)\geq h(x)(1+mv) +\frac{c_\lambda(x)}{m}(1+mv)\ln(1+mv).
		\end{align*}
	\end{lemma}
	
	By Lemma \ref{exponentialchange} we have,
	\begin{align}\label{w_i}
		-\divi(A(x)Dw_i)&=(1+\nu_iw_i)\left[c_\lambda(x)g_i(w_i)+h(x)+\big([M(x)-\nu_iA(x)]Du,Du\big)\right].
	\end{align}
	Note that the last term is negative for $i=1$ and positive for $i=2$.
	Using \eqref{w_i} we shall obtain a uniform a priori upper bound on $u$ in a neighborhood of any fixed point $\overline{x}\in\overline{\Omega}_{c^+}$. We consider the two cases $\overline{x}\in\overline{\Omega}_{c^+}\cap\Omega$ and $\overline{x}\in\overline{\Omega}_{c^+}\cap\partial\Omega$ separately.
	
	\begin{lemma}\label{interior}
		Assume that \eqref{A} holds and that $\overline{x}\in\overline{\Omega}_{c^+}\cap\Omega$. For each $\Lambda_2>\Lambda_1>0$, there exist $M_1>0$ and $R>0$ such that, for any $\lambda\in[\Lambda_1,\Lambda_2]$, any solution $u$ to \eqref{$P_lambda$} satisfies $\sup\limits_{B_R(\widetilde{x})}u\leq M_1$.
	\end{lemma}
	\begin{proof}
		Under the assumption \eqref{A} we can find a $R>0$ such that $M(x)\geq \mu_1I_n>0, \ c^-\equiv 0$ in $B_{4R}(\widetilde{x})$ and $c^+\gneqq 0$ in $B_{R}(\overline{x})$.
		Observe that from \eqref{w_i} for $i=1$, we get
		\begin{align*}
			-\divi(A(x)Dw_1)
			&\geq(1+\nu_1w_1)[\lambda c^+(x)g_1(w_1)+h^+(x)]-h^-(x)-\nu_1h^-(x)w_1\\
			&+(1+\nu_1w_1)(\mu_1-\vartheta^{-1}\nu_1)|Du|^2.
		\end{align*}
		Therefore, in $B_{4R}(\overline{x})$ it yields,
		\begin{align}\label{6.5}
			-\divi(A(x)Dw_1)+\nu_1h^-(x)w_1&\geq(1+\nu_1w_1)[\lambda c^+(x)g_1(w_1)+h^+(x)]-h^-(x).
		\end{align}
		Let $z_0$ be the solution to
		\begin{align}\label{6.6}
			-\divi(A(x)Dz_0)+\nu_1h^-(x)z_0&=-\Lambda_2c^+(x)\frac{e^{-1}}{\nu_1},\mbox{ } z_0\in H_0^1(B_{4R}(\widetilde{x})).
		\end{align}
		By the classical regularity, see \cite[Theorem III-14.1]{MR0244627}, $z_0\in C(\overline{B_{4R}(\widetilde{x})})$ and there exists a positive constant $\overline{C}= \overline{C}(\overline{x},\nu_1, \Lambda_2, p, R, \|h^-\|_{L^p(B_{4R})}, \|c^+\|_{L^p(B_{4R})})$ such that
		$ \displaystyle z_0\geq-\overline{C} \mbox{ in } B_{4R}$. 
		Further, by the Weak Maximum Principle 
		we know that $z_0\leq 0$. Since
		\begin{align*}
		\min_{(-\scriptstyle\frac{1}{\nu_i},\infty)}(1+\nu_is)g_i(s)=-\frac{e^{-1}}{\nu_i}, \quad \text{setting} \quad v_1:=w_1-z_0+\frac{1}{\nu_1} \quad \text{it \ satisfies},
		\end{align*}
		\begin{equation}
			\begin{array}{rcl}\label{v_1}
				\displaystyle-\divi(A(x)Dv_1)+\nu_1h^-(x)v_1
				&\geq&(1+\nu_1w_1)[\lambda c^+(x)g_1(w_1)+h^+(x)]+\Lambda_2c^+(x)\displaystyle\frac{e^{-1}}{\nu_1}\\
				\displaystyle&\geq& (1+\nu_1w_1)(\Lambda_2- \Lambda_2)[ c^+(x)g_1^-(w_1)]\\
				\displaystyle &+& (1+\nu_1w_1)\Lambda_1 c^+(x)g_1^+(w_1)\\
				\displaystyle&\geq&\displaystyle\frac{\Lambda_1 c^+(x)}{\nu_1}(1+\nu_1w_1)\ln(1+\nu_1w_1)\\
				&=&f(x,v_1), \quad \text{in} \quad B_{4R}(\overline{x}),
			\end{array}
		\end{equation}	
		\begin{eqnarray}\label{f}
		\text{where} \qquad  \qquad \qquad		f:\Omega\times\mathbb{R}&\rightarrow&\mathbb{R}\\\nonumber
				(x,s)&\rightarrow& f(x,s):=\Lambda_1c^+(x)\big([s+z_0][\ln(\nu_1)+\ln(s+z_0(x))] \big)
		\end{eqnarray}
		is a superlinear function in the variable $s$. Since $w_1>-1/\nu_1$, we have $v_1>0$ in $\overline{B_{4R}(\overline{x})}$.
	
		On the other hand, for $i=2$, in view of  \eqref{5.2} and $w_2>-1/\nu_2$, by \eqref{w_i} in a similar way we conclude that $w_2$ satisfies
		\begin{equation}
			\begin{array}{rrl}\label{w_2}
				\displaystyle-\divi(A(x)Dw_2)&\leq&[1+\nu_2w_2](\lambda c^+(x)g_2(w_2)+h^+(x))\\
				&+& (\nu_1-\nu_2)h^-(x)w_2-h^-(x)-\nu_1h^-(x)w_2 \quad \text{and}\\
				\displaystyle-\divi(A(x)Dw_2)+\nu_1h^-(x)w_2&\leq&[1+\nu_2w_2]\left(\displaystyle\frac{\Lambda_2 c^+(x)}{\nu_2}\ln(1+\nu_2w_2)+h^+(x)\right)\\
				\displaystyle&=:& g(x,w_2) \quad \text{in} \quad B_{4R}(\overline{x}),
			\end{array}
		\end{equation}
		 where $g:\Omega\times\mathbb{R}\rightarrow\mathbb{R}$ satisfies
		\begin{align}\label{g}
			g(x,s)\leq a_0[1+(\nu_2s)^{\alpha+1}], \quad\mbox{for each} \quad \alpha>0, \quad \mbox{where \ } a_0(x)\in L^p(\Omega).
		\end{align}
		In fact, in order to obtain \eqref{g}, let $c_\alpha>0$ be a constant such that
		\begin{align*}
			\ln(1+x)&\leq (1+x)^\alpha+c_\alpha, \mbox{ for all } x\geq 0.\\
			\mbox{Then,} \qquad \quad	g(x,w_2)
			\leq[1+\nu_2w_2]&\left(\frac{\Lambda_2}{\nu_2}c^+(x)(1+\nu_2w_2)^\alpha+c_\alpha\frac{\Lambda_2}{\nu_2}c^+(x)+h^+(x)\right)\\
			\leq [1+\nu_2w_2]&^{\alpha+1}\left(\frac{\Lambda_2}{\nu_2}c^+(x)(1+c_\alpha)+h^+(x)\right)\leq[1+(\nu_2w_2)^{\alpha+1}]a_0(x).
		\end{align*}
			In addition, we note that $[1+\nu_2w_2]^{\frac{\nu_1}{\nu_2}}=(e^{\nu_2u})^{\frac{\nu_1}{\nu_2}}=(e^{\nu_1u})=1+\nu_1w_1=\nu_1[v_1+z_0],$
	which means that $w_2=\xi(v_1+z_0)$, where $\xi(s):=[(\nu_1s)^{\frac{\nu_2}{\nu_1}}-1]\nu_2^{-1}$ is an increasing function satisfying
		\begin{align}\label{xi}
			\lim_{s\rightarrow\infty}\frac{\xi(s)}{s^\beta}=\lim_{s\rightarrow\infty}\frac{(\nu_1s)^{\nu_2/\nu_1}-1}{\nu_2s^{\nu_2/\nu_1}}=\lim_{s\rightarrow\infty}\frac{\nu_1^{\nu_2/\nu_1}-\frac{1}{s^{\nu_2/\nu_1}}}{\nu_2}=\frac{\nu_1^{\nu_2/\nu_1}}{\nu_2}<\infty, \ \mbox{for \ }  \beta=\nu_2/\nu_1.
		\end{align}		
		Thus, we are in position to apply the following theorem, which under our assumptions is a straightforward generalization of  \cite[Theorem 2]{Anew}. In fact, as a consequence of the Boundary Weak Harnack Inequality, Theorem \ref{bwhi} the Theorems 3-6 stated in \cite{Anew} are valid under our assumptions
		on the domain and on the coefficients of \eqref{$P_lambda$}. Hence, it remains to observe that the other
		generalizations on the hypotheses of Theorem 5.3, in comparison with  \cite[Theorem 2]{Anew}, are natural,
		in view of  \cite[Remark 4]{Anew}. As a matter of completeness, we state here our adapted version. 
		
		\begin{theorem}\label{newmethod}
			Let $\Omega\subset\mathbb{R}^n$, $n\geq 2$ be a $C^{1,\mathcal{D}ini}$  bounded domain and consider the coefficients of the problems $(P_\lambda)$ under our assumptions. Assume that $z_0$ is a bounded function, $v\geq 0$ and $\xi(v+z_0)$, where $\xi$ satisfies \eqref{xi}, are functions in $H^{1}(\Omega)$ satisfying the following inequalities in the weak sense
			\begin{align*}
				-\divi(A(x)Dv)+\nu_1h^{-1}(x)v&\geq f(x,v)\\
				-\divi(A(x)D\xi(v+z_0))+\nu_1h^{-1}(x)\xi(v+z_0)&\leq g(x,\xi(v+z_0)),
			\end{align*}
			where $f$  satisfies \eqref{f} and $g$ satisfies \eqref{g}  for some $r=\alpha+1$ with
			\begin{align*}
				r<\frac{n+1}{n-1}+\left(\frac{1}{\beta}-1\right)\frac{2}{n-1}.
			\end{align*}
			Then, for some $C$ depending on the concerned quantities we have
			\begin{align*}
				\xi(v(x)+z_0)\leq Cd(x) \mbox{ \ in \ } \Omega \quad \mbox{ and \ hence }\quad  v(x)\leq C.
			\end{align*}
		\end{theorem}		
		In view of \eqref{v_1} and \eqref{w_2} we are able to apply Theorem \ref{newmethod} for $v=v_1$ and $w_2=\xi(v_1+z_0)$ and conclude that $v_1$ and $w_2$ have upper bounds in $B_{4R}(\overline{x})$. As a consequence, the same holds for $w_1$ and also for $u$, as desired.
	\end{proof}
	\begin{lemma}\label{exterior}
		Assume that \eqref{A} holds and that $\overline{x}\in\overline{\Omega}_{c^+}\cap\partial\Omega$. For each $\Lambda_2>\Lambda_1>0$, there exist $R>0$ and $M_2>0$ such that, for any $\lambda\in[\Lambda_1,\Lambda_2]$, any solution $u$ to \eqref{$P_lambda$} satisfies $\sup\limits_{B_R(\overline{x})\cap\Omega}u\leq M_2$.
	\end{lemma}
	\begin{proof}
		This proof is very similar to the previous one, we only need to observe that our assumptions allow us to find $\Omega_1\subset\Omega$ with $\partial\Omega_1$ of class $C^{1,\mathcal{D}ini}$ such that $B_{2R}(\overline{x})\cap \Omega\subset\Omega_1$ and $M(x)\geq \mu_1 I_n>0$, $c^-(x)\equiv 0$ and $c^+(x)\gneqq 0$ in $\Omega_1$.
		Hence, for $i=1$ note that \eqref{w_i} turn into \eqref{6.5} in $\Omega_1$ instead of $B_{4R}(\overline{x})$. Then, if $z_0$ is the solution to \eqref{6.6} in $H^1_0(\Omega_1)$ instead of $H^1_0(B_{4R}(\overline{x}))$, as in Lemma \ref{interior}, we get $z_0\in C(\overline{\Omega}_1)$ and $\overline{C}>0$ depending on the usual quantities such that $-\overline{C}\leq z_0\leq 0$ in $\Omega_1$.
		In addition, defining $v_1$ as before, we observe  that $v_1$ satisfies the equation \eqref{v_1} in $\Omega_1$ and $v_1>0$ on $\overline{\Omega}_1$. Therefore, arguing exactly as Lemma \ref{interior} we deduce \eqref{v_1},\eqref{w_2} and then we are able to apply Theorem \ref{newmethod} getting an upper bound to $u$ in $\Omega_1$.
	\end{proof}
	
	\begin{proof}[Proof of Theorem \ref{6.3}] In view of the Lemmas \ref{interior} and \ref{exterior} we have the existence of a uniform a priori upper bound on $u$ in a neighborhood of any fixed point $\overline{x} \in\overline{\Omega}_{c^+}$. Then, by applying a topological approach relying on the derivation of a priori bounds, this proof follows the same lines as the proof of the Interior Weak Harnack inequality - IWHI, for details see \cite{MR4030257}.
	\end{proof}
	We will now see that, under our assumptions, any solution to problem \eqref{$P_lambda$} is bounded from below, even when $\lambda\to 0$, $\lambda>0$.
	\begin{theorem}[\bf A Priori Lower Bound]\label{lowerbound}
		Under the standing assumptions on problem \eqref{$P_lambda$}, including hypothesis \eqref{A},  let $\Lambda_2>0$.
		Then, every supersolution $u$ to \eqref{$P_lambda$} satisfies
		\begin{align*}
			\|u^-\|_{L^\infty}\leq C \mbox{ for all } \lambda\in [0,\Lambda_2], \quad 	\mbox{where\ } C= C(n,p,\nu_1, \Omega, \Lambda_2,\|c\|_{L^p(\Omega)},\|h^-\|_{L^p(\Omega)}).
		\end{align*}
	\end{theorem}
	\begin{proof}
		First observe that both $U_1=-u$ and $U_2=0$ are subsolution of
		\begin{align*}
			-\divi(A(x)DU)\leq c_\lambda U -(M(x)DU,DU)+h^-(x) \mbox{ in }\Omega.
		\end{align*}
		Then, these functions are also subsolutions of
		\begin{eqnarray*}
			\left\{
			\begin{array}{rll}
				-\divi(A(x)DU)+ \mu_1|Du|^2&\leq c_\lambda U +h^-(x) &\mbox{ in } \Omega\\
				U&\leq 0 &\mbox{ on } \partial\Omega
			\end{array}
			\right.
		\end{eqnarray*}
		and so is $U:=u^-=\max\{U_1,U_2\}$, as the maximum of subsolutions. Moreover, $U\geq 0$ in $\Omega$ and $U=0$ on $\partial\Omega$.
		We make the following exponential change of variables
		$			w:=\dfrac{1-e^{-\nu_1U}}{\nu_1}.
	$
		From Lemma \ref{exponentialchange},
		\[-\divi(A(x)Dw)\leq (1-\nu_1 w) \left[c_\lambda(x) U +h^-(x) \right],\]
	hence,	we know that $w$ is a weak solution to
		\begin{align}\tag{$Q_\lambda$}\label{w}
			\left\{
			\begin{array}{rll}
				-\divi(A(x)Dw)+\nu_1h^-(x)w&\leq  h^-(x)+\dfrac{c_\lambda(x)}{\nu_1}\ln (1-\nu_1w)(1-\nu_1w)  &\mbox{ in } \Omega\\
				w&= 0 &\mbox{ on } \partial\Omega.
			\end{array}
			\right.
		\end{align}
		Now set $w_1:=\textstyle\dfrac{1-e^{-\nu_1u_1^-}}{\nu_1}$, where $u_1$ is some fixed supersolution to \eqref{$P_lambda$}, $\lambda\geq 0$, Note that if there was not such supersolution, we had nothing to prove. Then, from the above conclusions, $w_1\in[0,1/\nu_1)$ is a solution to \eqref{w}.
		Defining
		\begin{align*}
			\overline{w}:=\sup \mathcal{A}, \mbox{ where } \mathcal{A}:=\{w: w \mbox{ is a solution to \eqref{w}; } 0\leq w <1/\nu_1 \mbox{ in } \Omega \},
		\end{align*}
		we observe that $\mathcal{A}\ne\emptyset$, since $w_1\in \mathcal{A}$, and also that $w_1\leq \overline{w}\leq 1/\nu_1$ in $\Omega$. Further, as a supremum of supersolutions, $\overline{w}$  is a supersolution, with $\overline{w}=0$ on $\partial \Omega$, i.e. $\overline{w}$ is a weak solution to \eqref{w}.
		Then,
		\begin{align*}
			f(x)&:=h^-(x)+\frac{c_\lambda(x)}{\nu_1}|\ln (1-\nu_1\overline{w})|(1-\nu_1\overline{w})\in L^p_+(\Omega)\\
			\mbox{with }
			\|f^+\|_{L^p(\Omega)}&\leq\|h^-\|_{L^p(\Omega)}+\frac{1}{\nu_1}\left(\Lambda_2\|c^+\|_{L^p(\Omega)}+\|c^-\|_{L^p(\Omega)}\right)C_0,
		\end{align*}
		since $A(\overline{w}):=|\ln (1-\nu_1\overline{w})|(1-\nu_1\overline{w})\leq C_0$.
		Therefore, by applying the Boundary Lipschitz Bound \cite[Lemma 2.14]{fio}, we conclude that
		\begin{align*}
			\overline{w}\leq C\|f^+\|_{L^p(\Omega)}d(x)\rightarrow 0 \mbox{ as }x\rightarrow\partial\Omega,
		\end{align*}
	since $d(x)$ denotes the distance between $x$ and $\partial \Omega$. Hence, $\overline{w}\not\equiv 1/\nu_1$, however,  $\overline{w}$ may be equal to $1/\nu_1$ at some interior points.
		In order to complete the proof, we argue by contradiction. Assume that there is a sequence of supersolutions $u_k$ to \eqref{$P_lambda$} in $\Omega$ with unbounded negative parts, then there exists a subsequence such that
		\begin{align*}
			u_k^-(x_k)=\|u_k^-\|_{L^\infty}\rightarrow +\infty, \quad x_k\in \overline{\Omega}, \quad x_k\rightarrow x_0\in\overline{\Omega} \mbox{\ as \ }k\rightarrow\infty,
		\end{align*}
		with $x_k\in \Omega$ for large $k$, since $u_k\geq 0$ on $\partial\Omega$. It implies that the respective sequence $(w_k(x_k))$ satisfies
		\begin{align*}
			w_k(x_k)=\frac{1-e^{-\nu_1 u_k^-(x_k)}}{\nu_1}\rightarrow \frac{1}{\nu_1}, \qquad w_k\in \mathcal{A}.
		\end{align*}
	Hence, for every $\varepsilon>0$, there exists some $k_0\in \mathbb{N}$ such that
		\begin{align*}
			\frac{1}{\nu_1}-\varepsilon \leq w_k(x_k)\leq \overline{w}(x_k)\leq \frac{1}{\nu_1},\mbox{ for all } k\geq k_0.
		\end{align*}
		Thus,
	$
		\overline{w}(x_0)\geq \displaystyle\lim_{x_k\rightarrow x_0}\overline{w}(x_k)=\displaystyle\lim_{k\rightarrow \infty} \overline{w}(x_k) ={1}/{\nu_1},
	$
	and $x_0 \in\Omega$, since $\overline{w}=0$ on $\partial\Omega$. Moreover, $\overline{w}(x_0)={1}/{\nu_1}$.
		Finally, set $z:=1-\nu_1\overline{w}$ and observe that
		\begin{align*}
			\divi(A(x)Dz)&=-\nu_1 \divi(A(x)D\overline{w})\leq \nu_1(1-\nu_1\overline{w}) \left[\frac{c_\lambda(x)}{\nu_1} |\ln(1-\nu_1\overline{w})| +h^-(x) \right] \\
			&= c_\lambda(x)|\ln z|z +\nu_1h^-(x)z.
		\end{align*}
		Then $z$ is a supersolution to
		\begin{align*}
			\left\{
			\begin{array}{rll}
				-\divi(A(x)Dz)+\nu_1h^-(x)z&\geq  -c_\lambda(x)|\ln z|z &\mbox{ in } \Omega\\
			z(x_0)=0 \quad \text{and} \quad	z&\gneqq 0 &\mbox{ in } \Omega.\\
			\end{array}
			\right.
		\end{align*}
		But this contradicts the nonlinear version of the SMP \cite{MR4274882}[Theorem 1.1], see also \cite[Lemma 5.3]{multiplicidade}, for equations in nondivergence form,
		which says that  $z\equiv 0$ or $z>0$ in $\Omega$.
	\end{proof}
	
	\section{Main Results}\label{results}
\quad \	This section is devoted to prove  our main results. We start by proving a lemma, which is going be useful in order to deal with degree arguments.
	
	\begin{lemma}\label{existssub}
		Under assumption \eqref{A} for every $\lambda>0$, there exists a strict subsolution $v_\lambda$ to \eqref{$P_lambda$} such that, every supersolution $\beta$ to \eqref{$P_lambda$} satisfies $v_\lambda\leq\beta$.
	\end{lemma}
	
	\begin{proof}
		Let $C>0$ be given by Theorem \ref{lowerbound} and $\overline{M}$ be given by Theorem \ref{6.3} such that for every supersolution $\beta$ of
		\begin{align*}
			\left\{
			\begin{array}{rll}
				-\divi(A(x)Du)& = c_\lambda(x)u+(M(x)Du,Du)-h^-(x)-1 &\mbox{ in } \Omega\\
				u&=0 &\mbox{ on } \partial\Omega,
			\end{array}
			\right.
		\end{align*}
		we have $\beta\geq-C$.
		Let $k>C$ and consider $\alpha_k$ the solution to
		\begin{align*}
			\left\{
			\begin{array}{rll}
				-\divi(A(x)Dv)+c^-(x)v& = -\lambda k c^+(x)-h^-(x)-1 &\mbox{ in } \Omega\\
				v&=0 &\mbox{ on } \partial\Omega.
			\end{array}
			\right.
		\end{align*}
		As $-\lambda k c^+(x)-h^-(x)-1<0$ we have $\alpha_k\ll 0 $ by the SMP and the Hopf lemma.
		
		We claim that every supersolution $\beta$ to \eqref{$P_lambda$} satisfies $\beta\geq \alpha_k$. In fact, taking regular supersolutions $\beta_1,\cdots,\beta_l$ to \eqref{$P_lambda$} such that $\beta=\min\{\beta_j: 1\leq j\leq l\}$ and
		setting $w=\beta_j-\alpha_k$ for some $1\leq k\leq l$ we have
		\begin{align*}
			\left\{
			\begin{array}{rll}
				-\divi(A(x)Dw)+c^-(x)w&\geq \lambda c^+(x)(\beta_j+k)+\mu_1|D\beta_j|^2\geq 0 &\mbox{ in } \Omega\\
				w&=0 &\mbox{ on } \partial\Omega.
			\end{array}
			\right.
		\end{align*}
		Hence, by the Maximum Principle $w\geq 0 $ i.e. $\beta_j\geq \alpha_k$ and this proves the claim.
		
		Now, consider the problem
		\begin{align}\label{t_k}
			-\divi(A(x)Dv)= &c_\lambda(x)T_k(v)+(M(x)Dv,Dv)-h^-(x)-1,\\
			\text{where} \qquad \qquad \qquad \qquad
			T_k(v)&=
			\left\{
			\begin{array}{rll}
				-k, & \mbox{ if } v\leq -k,\\
				v,& \mbox{ if } v>-k.
			\end{array}
			\right. \nonumber \qquad\qquad \qquad \qquad \qquad\qquad\qquad\qquad
		\end{align}
	
		Observe that $\beta=T_k(\beta)$  is a supersolution to \eqref{t_k}  and $\alpha_k$ is a subsolution to \eqref{t_k}. Note that $-\lambda k c^+(x)=\lambda c^+(x) T_k(\alpha_k)$, 
		$c^-(x)k=-c^-(x)T_k(\alpha_k)$ and hence, by the standard method of sub and supersolutions \eqref{t_k} has a minimal solution $v_k$ with $\alpha_k\leq v_k\leq \beta$. Furthermore, we also observe  that every supersolution $\widetilde \beta$ to \eqref{$P_lambda$} satisfies $\widetilde \beta\geq v_k$.
		In fact, since $\widetilde \beta$ is a supersolution to \eqref{$P_lambda$}, we have $\widetilde \beta \geq \alpha_k$ and by the construction of \eqref{t_k}, every supersolution $\widetilde \beta$ to \eqref{$P_lambda$} is also a supersolution to \eqref{t_k}, then the minimality  of $v_k$ implies that $v_k\leq \widetilde \beta$.
		
		Now we observe that $v_k$ is a subsolution to \eqref{$P_lambda$}, since $v_k\geq -C>-k$ and it satisfies
		\begin{align*}
			-\divi(A(x)Dv_k)&=c_\lambda(x)T_k(v_k)+(M(x)Dv_k,Dv_k)-h^-(x)-1\\
			&\leq  c_\lambda(x)v_k+(M(x)Dv_k,Dv_k)+h(x).
		\end{align*}
		Furthermore, we claim that $v_k$ is strict subsolution to \eqref{$P_lambda$}.
		In order to see this, let $u$ be a solution to \eqref{$P_lambda$} with $u\geq v_k$.
		Then, $w=u-v_k$ satisfies
		\begin{align*}
			-\divi(A(x)Dw)&\geq c_\lambda(x)u+(M(x)Du,Du)+h(x)-c_\lambda(x)v_k-(M(x)Dv_k,Dv_k)+h^-(x)+1\\
			&=c_\lambda(x)w+(M(x)[Du+Dv_k],Dw)+h^+(x)+1,
		\end{align*}
		which means that
		\begin{align*}
			\left\{
			\begin{array}{rll}
				-\divi(A(x)Dw) -(M(x)[Du+Dv_k],Dw) &\geq c_\lambda(x)w+h^+(x)+1 &\mbox{ in } \Omega\\
				w&=0 &\mbox{ on } \partial\Omega.
			\end{array}
			\right.
		\end{align*}
	Therefore, by the Maximum Principle, we deduce that $w\gg 0$, namely, $u\gg v_k$.
	\end{proof}
	
	By adapting \cite[Lemma 5.1]{MR4030257} to our setting, we obtain the following auxiliary result, which is going to be useful for proving Theorem \ref{teo5.2}.
	\begin{lemma}\label{nosolution}
		Under the assumptions of Theorem \ref{teo5.2}, assume that ($P_0$) has a solution $u_0$ such that $c^+(x)u_0\gneqq 0$. Then, there exists $\overline{\Lambda}\in(0,\infty)$ such that, for $\lambda\geq \overline{\Lambda}$, the problem \eqref{$P_lambda$} has no solution $u$ with $u\geq u_0$ in $\Omega$.
	\end{lemma}
	\begin{proof}
		Let $\varphi_1>0$ the first eigenfunction of \eqref{eig1}. If \eqref{$P_lambda$} has a solution $u$ with $u\geq u_0$, multiplying
		\eqref{$P_lambda$} by $\varphi_1$ and integrating we obtain
		\begin{align*}
			\displaystyle\int_{\Omega}c_{\gamma_1}(x)u\varphi_1dx&= \displaystyle\int_{\Omega}A(x)D\varphi_1 Dudx=\displaystyle\int_{\Omega} c_\lambda(x)u\varphi_1 dx+\int_{\Omega}(M(x)Du,\varphi_1 Du)dx +\int_{\Omega} h(x)\varphi_1 dx,
		\end{align*}
		and hence $\lambda>\gamma_1$. As $u\geq u_0$, we have
		\begin{align*}
			0&\geq (\lambda-\gamma_1)\displaystyle\int_{\Omega} c^+(x)u\varphi_1 dx + \mu_1\int_{\Omega}\varphi_1 |Du|^2 dx+\int_{\Omega} h(x)\varphi_1 dx\\
			&\geq (\lambda-\gamma_1)\displaystyle\int_{\Omega} c^+(x)u_0\varphi dx + \mu_1\int_{\Omega}\varphi_1 |Du|^2 dx+\int_{\Omega} h(x)\varphi_1 dx,
		\end{align*}
		which gives a contradiction for $\lambda$ large enough.
	\end{proof}

	In view of the previous results, we now are able to prove Theorem \ref{teo5.2}.
	
	\begin{proof}[Proof of Theorem \ref{teo5.2}]
	Applying all previous results and adopting strategies presented in \cite{ACJT, ACJTuni, MR4030257} we give the proof of Theorem \ref{teo5.2} treating separately the cases $\lambda\leq 0$ and $\lambda>0$.\medskip
	\newline
	{\bf Case (i): $\lambda\leq 0$.} This case has been studied in previous works. We briefly recall the following argument. If ($P_0$) has a solution $u_0$, then $u_0$ is a supersolution to  \eqref{$P_lambda$}. By applying Lemma \ref{lbound} and the arguments found in \cite{ACJT}, we obtain the existence of a solution $u_\lambda$ of (\ref{$P_lambda$}) for any $\lambda<0$. We observe that the uniqueness of solutions for $\lambda\leq 0$ is ensured by \cite[Proposition 4.1]{ACJT}.
	On the other hand, for $\lambda \leq 0$, we have $c_\lambda(x)=\lambda c^+(x)-c^-(x)\leq -c^-(x)$ so by applying the Comparison Principle, we get $u_\lambda\leq u_0$. Moreover, setting $v_0=u_0-\|u_0\|_{\infty}$, by Lemma \ref{lbound} we see that $v_0$ is a subsolution to \eqref{$P_lambda$} for $\lambda<0$, so again by the Comparison Principle we get $u_0-\|u_0\|_{\infty}\leq u_\lambda.$ \medskip
	\newline{\bf Case (ii): $\lambda>0$}.
	With the aim of showing the existence of a continuum of solution to \eqref{$P_lambda$}, for $\lambda\geq 0$ we introduce the auxiliary problem
	\begin{equation}\label{P2}\tag{$\overline{P_\lambda}$}\begin{cases}
		-\divi(A(x)Du)+u&=[c_\lambda(x)+1][(u-u_0)^++u_0]+\big(M(x)Du,Du\big)+h(x) \quad \mbox{ in } \Omega \\
\qquad \qquad \qquad \qquad  u &= 0 \quad \mbox{ on } \partial \Omega.
\end{cases}	\end{equation}
	In view of  \cite{DE}, more specifically,   \cite[Theorem 3,1]{RSS}, we know that, as in the case of problem (\ref{$P_lambda$}), any solution to problem \eqref{P2} belongs to $C^{1,\mathcal{D}ini}(\Omega)$.	
	Moreover, observe that $u$ is a solution to \eqref{P2} if and only if it is a fixed point of the operator $\overline{T}_\lambda$ defined by $\overline{T}_\lambda:C^1(\overline{\Omega})\rightarrow C^1(\overline{\Omega}):v\rightarrow u$ with $u$ the solution to
	\begin{align*}
		-\divi(A(x)Du)+u-\big(M(x)Du,Du\big)=[c_\lambda(x)+1][(v-u_0)^++u_0]+h(x).
	\end{align*}
	Applying  \cite[ Lemma 5.2]{ACJT}, we see that $\overline{T}_\lambda$ is completely continuous. Now, we define
	\begin{align*}
		\overline{\Sigma}:=\{(\lambda,u)\in\mathbb{R}\times \mathcal{C}(\overline{\Omega}), u \mbox{ solves \eqref{P2}}   \} \ \mbox{and split the rest of the proof into three steps.}
	\end{align*}
	{\bf Step 1:} {\it If $u$ is a solution to \eqref{P2}, then $u\geq  u_0$ and hence, it is a solution to \eqref{$P_lambda$}.}	Observe that $(u-u_0)^++u_0-u\geq 0$ and $\lambda c^+(x)[(u-u_0)^++u_0]\geq \lambda c^+(x)u_0\geq 0$.
	Then, we deduce that a solution $u$ to \eqref{P2} is a supersolution to
	\begin{equation}\label{7.1}
		-\divi(A(x)Du)=[c_\lambda(x)+1][(u-u_0)^++u_0]+\big(M(x)Du,Du\big)+h(x).
	\end{equation}
	Since $u_0$ is a solution to \eqref{$P_lambda$}, it implies that $u_0$ solves \eqref{7.1}. Thus, applying again the Comparison Principle  we get $u\geq u_0$.\medskip
	\newline
	{\bf Step 2:} {\it $u_0$ is the unique solution to ($\overline{P_0}$) as well as to the problem ($P_0$) and $i(I-\overline{T}_0,u_0)=1$.} For $\lambda=0$, if  $u$ is a solution to (\ref{7.1}), then by Step 1, $u\geq u_0$ and $u$ solves (\ref{$P_lambda$}). From Case (i), we conclude that $u=u_0$.
	In order to prove that $i(I-\overline{T}_0,u_0)=1$, we consider the operator  $S_t$ defined by
	$	S_t:C^1(\overline{\Omega})\rightarrow C^1(\overline{\Omega})$, given by  $S_t(v)=t\overline{T}_0v=u
$, where $u$ is the solution to
{\small	\begin{align*}
		\textstyle-\divi(A(x)Du)+u&=(M(x)Du,Du)+th(x) +t\big([-c^-(x)+1][u_0+(v-u_0)^+-(v-u_0-1)^+]\big).
	\end{align*}}
First, note that the complete continuity of $\overline{T}_\lambda$ follows from the fact that every solution $u$ to \eqref{P2} is $C^{\alpha}$ up to the boundary, and hence there exists $R>0$ such that for all $t \in [0,1]$ and all $v \in C^1(\overline{\Omega})$, it follows that
	$	\|S_tv\|_{L^\infty}<R.$
	Then, $I-S_t$ does not vanish on $\partial B_R(0)$ and
$$
	{	\deg(I-\overline{T}_0,B_R(0))= \deg(I-S_1,B_R(0))=\deg(I-S_0,B_R(0))=\deg(I,B_R(0))=1.}
$$
	Therefore, $\overline{T}_0$ has a fixed point $u_0$, which is a solution to $(\overline{P}_0)$. Applying the degree's properties, for all $\varepsilon>0$ small enough, it follows that
	$
	\displaystyle	\deg(I-\overline{T}_0,B_\varepsilon(0))=\deg(I-\overline{T}_0,B_R(0))=1.
	$
	Thus, for $\varepsilon \ll 1$, we conclude that
	$
		\displaystyle i(I-\overline{T}_0,u_0)=\lim_{\varepsilon\rightarrow 0}\deg(I-\overline{T}_0,B_\varepsilon(0))=1.
	$
	\medskip
	\newline
	{\bf Step 3:} {\it Existence and behavior of the continuum.} 	Proceeding as in \cite[Theorem 1.2]{CJ}, we are able to apply Theorem \ref{continuum}, which gives us a continuum $\mathcal{C}=\mathcal{C}^+\cup \mathcal{C}^-\subset \overline{\Sigma}$ such that
	\begin{align*}
		\mathcal{C}^+=\mathcal{C}\cap ([0,\infty)\times C(\overline{\Omega}))\mbox{ and } \mathcal{C}^-=\mathcal{C}\cap ((-\infty,0]\times C(\overline{\Omega})) \ \mbox{ are unbounded in} \ \mathbb{R}^{\pm}\times C(\overline{\Omega}).
	\end{align*}
	By Step 1, if $u \in \mathcal{C}^+$, then $u\geq u_0$ and it is a solution to \eqref{$P_lambda$}. Thus, applying Lemma \ref{nosolution} we infer that the projection of  $\mathcal{C}^+$ on $\lambda$-axis is $[0,\overline{\Lambda}]$, a bounded interval.
	A consequence  of Case (i) is that none of $\lambda\in (-\infty,0] $ is a bifurcation point from infinity to problem (\ref{$P_lambda$}), and then, we deduce that the projection of $\mathcal{C}^-$ on $\lambda$-axis is $(-\infty,0]$.
	Hence,
	$
		\mbox{Proj}_{\mathbb{R}}\mathcal{C}=\mbox{Proj}_{\mathbb{R}}\mathcal{C}^-\cup\mbox{Proj}_{\mathbb{R}}\mathcal{C}^+=(-\infty,\overline{\Lambda}],$ for  some $\overline{\Lambda}>0.
	$
	Finally, by Theorem \ref{6.3} for any $0<\Lambda_1<\Lambda_2$ there exists a priori bound for the solutions to \eqref{$P_lambda$}, for all $\lambda\in[\Lambda_1,\Lambda_2]$, then the projection of $\mathcal{C}\cap ([\Lambda_1,\Lambda_2]\times C(\overline{\Omega}))$ on $C(\overline{\Omega})$ is bounded.
	Since the component $\mathcal{C}^+$ is unbounded in $\mathbb{R}^+\times C(\overline{\Omega})$, its projection on the $C(\overline{\Omega})$ axis must be unbounded. In view of Case (i), the projection $\mathcal{C}^-$ on the $C(\overline{\Omega})$ is bounded.
	Hence,
$
		\mbox{Proj}_{C(\overline{\Omega})}\mathcal{C}=\mbox{Proj}_{C(\overline{\Omega})}\mathcal{C}^-\cup\mbox{Proj}_{C(\overline{\Omega})}\mathcal{C}^+=[0,+\infty).
$
	Therefore, we deduce that $\mathcal{C}$ must emanate from infinity on the right of axis $\lambda=0$.
	
	Now, we prove our multiplicity results in (iii). Since $\mathcal{C}$ contains $(0,u_0)$, with $u_0$ being the unique solution to ($P_0$),  from Case (ii) we know that $\mathcal{C}$ also emanates from infinity on the right of axis $\lambda=0$ and then, we conclude that there exists $\lambda_0\in (0,\overline{\Lambda})$ such that the problems (\ref{P2}) and \eqref{$P_lambda$} have at least two solutions satisfying $u\geq u_0$ for $\lambda\in(0,\lambda_0)$. Hence, the quantity
	\begin{align*}
		\overline{\lambda}:=\sup \{\mu>0: \forall \ \lambda\in(0, \mu), (P_\lambda) \mbox{ has at least two solutions} \} \quad \mbox{is well defined.}
	\end{align*}
	We claim that for all $\lambda\in (0, \overline{\lambda})$, the problem (\ref{$P_lambda$}) has at least two solutions with $u_{\lambda,1}\ll u_{\lambda,2}$.	Let us consider the strict subsolution $\alpha_\lambda$ given by Lemma \ref{existssub}. As $\alpha_\lambda\leq u$ for all $u$ solution to \eqref{$P_lambda$}, we can choose $u_{\lambda,1}$ as the minimal solution with $u_{\lambda,1}\geq \alpha$.
	Hence, we have $u_{\lambda,1}\lneqq u_{\lambda,2}$, otherwise there would exist a solution $u$ with $\alpha\leq u \leq \min\{u_{\lambda,1}, u_{\lambda,2}\}$, which contradicts the minimality of $u_{\lambda,1}$.
	Observe that, the function $\beta=(u_{\lambda,1}+u_{\lambda,2})/2$ is a supersolution to \eqref{$P_lambda$} which is not a solution.
	Now, for each $\xi\in \mathbb{R}^n$, we can define the function $\varphi(\xi):=(M(x)\xi,\xi)$ and observe that by assumption \eqref{A} we have $D^2(\varphi)>0$, and thus $\varphi$ is convex. With this, we obtain
	\begin{align*}
		-\divi(A(x)D\beta) &=-\frac{1}{2}\divi(A(x)Du_{\lambda,1})-\frac{1}{2}\divi(A(x)Du_{\lambda,2})\\
		&=c_\lambda(x)\beta+ \frac{1}{2}(M(x)Du_{\lambda,1},Du_{\lambda,1})+\frac{1}{2}(M(x)Du_{\lambda,2},Du_{\lambda,2})+h(x)\\
		&\gneqq c_\lambda(x)\beta+\varphi\big(\frac{Du_{\lambda,1}}{2}+\frac{Du_{\lambda,2}}{2}\big)+h(x) = c_\lambda(x)\beta+(M(x)D\beta,D\beta)+h(x).
	\end{align*}
	Let us prove that $\beta$ is a strict supersolution to (\ref{$P_lambda$}).
If  $u$ is a solution to (\ref{$P_lambda$}) with $u\leq \beta$, then $v:=\beta-u$ satisfies
	\begin{align*}
		-\divi(A(x)Dv)\gneqq c_\lambda(x)\beta&+(M(x)D\beta,D\beta)+h(x)-(M(x)Du,Du)-c_\lambda u -h(x)\\
		=(M(x)&[D\beta+Du],Dv)+c_\lambda v,\\
\text{and \ hence} \qquad \quad 	-\divi(A(x)Dv) &-(M(x)D\beta+Du,Dv)+c^-(x)v\gneqq \lambda c^+(x) v\geq 0. \quad \quad
	\end{align*}
	By Theorem \ref{SMP}, either $v\gg 0$ or $v\equiv 0$.
	If $v \equiv 0$, then $\beta=u$ is solution, which contradicts the definition of $\beta$.
	Then,  $\beta\gg u$. As $u_{\lambda,1}\lneqq \beta\lneqq u_{\lambda,2}$ we have, $u_{\lambda,1}\ll \beta\lneqq u_{\lambda,2}$ and hence  $u_{\lambda,1}\ll u_{\lambda,2}$.

	We finish the proof claiming that
if $\overline{\lambda}<\infty$, then the solution $u_{\overline{\lambda}}$ of ($P_{\overline{\lambda}}$) is unique.
In order to prove that ($P_{\overline{\lambda}}$) has at least one solution, take $\{\lambda_n\}\subset (0,\overline{\lambda})$ such that $\lambda_n\rightarrow \overline{\lambda}$ and by the regularity result \cite[Lemma 2.1]{ACJTuni} let $\{u_n \}\subset H^1(\omega)\cap W^{1,n}_{loc}(\Omega)\cap C(\overline{\Omega})$ be a sequence of corresponding solutions.
		By Theorem \ref{6.3}, there exists $M>0$ such that $\|u_n\|_{L^\infty}<M$ for all $n\in \mathbb{N}$, and hence in view of  the $C^{1,\mathcal{D}ini}$ global estimates \cite{DE}, we get $\|u_n\|_{C^{1,\mathcal{D}ini}(\overline{\Omega})}\leq C$.
		Then, up to a subsequence, $u_n\rightarrow u$ in $C^{1}_0(\Omega)$.
		From this strong convergence we easily observe that $u$ is a solution to  ($P_{\overline{\lambda}}$).
		
		Now we prove the uniqueness of the solution to  ($P_{\overline{\lambda}}$).
	We assume by contradiction that there exist two distinct solutions, $u_1$ and $u_2$ to problem ($P_{\overline{\lambda}}$), we prove that $\beta=(u_1+u_2)/2$ is a strict supersolution to  ($P_{\overline{\lambda}}$). Let us consider the strict subsolution $\alpha_{\overline{\lambda}}\ll \beta$ to problem ($P_{\overline{\lambda}}$) given by Lemma \ref{existssub} and look at the set
		\begin{align*}
			\overline{\mathcal{S}}=\{u\in C_0^1(\overline{\Omega})\mbox{; } \alpha\ll u\ll \beta,\|u\|_{C^1_0}<R\}
		\end{align*}
		for some $R>C>0$. Again, by the $C^{1,\mathcal{D}ini}$ estimates,
		\begin{align}\label{overlambda}
			\|u\|_{C^{1,\mathcal{D}ini}}\leq C \mbox{ for all } \mbox{ solution $u$ to } \eqref{$P_lambda$}\mbox{, } \lambda\in[\overline{\lambda}, \overline{\lambda}+1]
		\end{align}
		such that
		$deg(I-T_{\overline{\lambda}},\mathcal{\overline{\mathcal{S}}})=1$.
		Now, we prove the existence of $\varepsilon>0$ such that
		\begin{align}\label{existsvarepsilonover}
			deg(I-T_\lambda,\overline{\lambda})=1\mbox{, for all } \lambda\in[\overline{\lambda} ,\overline{\lambda}+\varepsilon].
		\end{align}
		There exists some $\varepsilon\in(0,1)$ such that there is no fixed points of $T_\lambda$ on the boundary of $\overline{\mathcal{S}}$ for all $\lambda$ in the preceding interval. Indeed, if this was not the case, there would exist a sequence $\lambda_k\rightarrow\overline{\lambda}$  with the respective solutions $u_k$ to problem ($P_{\lambda_k}$) belonging to $\overline{\mathcal{S}}$.
		Say $\lambda_k \in[\overline{\lambda} ,\overline{\lambda}+1]$  for $k \geq k_0$. Then, since $\alpha\ll u_k \ll \beta$ in $\Omega$, by \eqref{overlambda} we must have $u_k\in \partial \overline{\mathcal{S}}$ for $k\geq k_0$, which means that for each such $k$,
		\begin{align}\label{touchover}
			\max_{\overline{\Omega}}(\alpha-u_k)=0\mbox{ or }  \min_{\overline{\Omega}}(u_k-\beta)=0.
		\end{align}
		By \eqref{overlambda} and the compact inclusion $C^{0,\mathcal{D}ini}(\overline{\Omega})\subset \subset C(\Omega)$, we know that $u_k\rightarrow u$ in $\Omega$ for some $u\in C^1(\Omega)$, up to subsequences. Then, $u$ is a solution to ($P_{\overline{\lambda}}$) and  by taking the limit as $k\rightarrow+\infty$ in the corresponding  inequalities  for $u_k$, it follows that $\alpha \leq \beta$ in $\Omega$. Thus, $\alpha\ll u \ll \beta$ in $\Omega$, since $\alpha$ and $\beta$ are strict. Passing \eqref{touchover} to the limit, we obtain that $u(x)=\alpha(x)$ or $u(x)=\beta(x)$ for $x\in \overline{\Omega}$, which contradicts the definition of $\alpha\ll u\ll \beta$.
		Hence, for obtaining \eqref{existsvarepsilonover} it is just necessary to apply the homotopy invariance in $\lambda$ in the interval $[\overline{\lambda} ,\overline{\lambda}+\varepsilon]$. With \eqref{existsvarepsilonover} in hand, we repeat exactly the same argument used above to obtain the existence of a second solution $u_{\lambda,2}$ to problem \eqref{$P_lambda$} for all $\lambda\in [\overline{\lambda} ,\overline{\lambda}+\varepsilon]$, which contradicts the definition of $\overline{\lambda}$.
	\end{proof}

In order to prove Theorem \ref{teo5.3}, we start by constructing an auxiliary problem $(P_{\lambda,k})$, for which we can assume that there is no solution for large $k$. This is a typical but essential argument that allows us to find a second solution via degree theory, by homotopy invariance in $k$.
	Fix $\Lambda_2>0$ and recall that Theorem  \ref{lowerbound} gives us an a priori lower uniform bound $C_0$ such that
$u\geq -C_0,$ for every weak supersolution $u$ of the problem \eqref{$P_lambda$}, for all  $\lambda\in[0,\Lambda_2].$
	Consider, the problem
	\begin{align*}\label{P3}\tag{$P_{\lambda,{k}}$}
		\left\{
		\begin{array}{rll}
			-\divi(A(x)Du)=c_\lambda(x)u+(M(x)Du,Du)+h(x)+k\widetilde{c}(x) \mbox{ in } \Omega\\
			\qquad\qquad\qquad u= 0\qquad\qquad\qquad\qquad\qquad\qquad\qquad\qquad \mbox{ on } \partial\Omega
		\end{array}
		\right.
	\end{align*}
	for $k\geq 0$, $\lambda\in [0,\Lambda_2]$ and $\widetilde{c}$  defined as
	\begin{align}\label{5.5}
		\widetilde{c}(x):=\widetilde{c}_{\Lambda_2}(x)=h^-(x)+\Lambda_2C_0c^+(x)+\widetilde{M}c^-(x)+Bc^+(x)
	\end{align}
	with $B=\gamma_1/ \nu_1$, where $\gamma_1=\gamma_1^+>0$ is the first eigenvalue, with weight $c$, associated to the eigenfunction $\varphi_1\in W^{2,p}(\Omega)$, given by \eqref{eig1}.
	Note that every solution to \eqref{P3} is also a supersolution to \eqref{$P_lambda$} since $k\widetilde{c}(x)\geq 0$.
	Then, in virtue of (\ref{5.5}), for all $k \geq 1$, we have that
	\begin{align*}
		c_\lambda(x)u+h(x)+k\widetilde{c}(x)\geq -\Lambda_2C_0c^+(x)-\widetilde{M}c^-(x)-h^-(x)+\widetilde{c}(x)=Bc^+(x)\gneqq 0.
	\end{align*}
	
	We now derive some results about the solutions of problem \eqref{P3}.
	
	\begin{lemma}\label{c2} Under assumption \eqref{A}, assume that ($P_0$) has a solution $u_0\leq 0$ with ${c^+(x)u_0\lneqq 0}$. Then for each fixed $\Lambda_2>0$ and $\lambda\in [0,\Lambda_2]$, there exists $k\geq 0$ such that
		\begin{itemize}
			\item[(i)]  For all $k> 1$, the problem \eqref{P3} has no solutions;
			\item[(ii)]  For all $k\in(0,1) $,  \eqref{P3} has at least two solutions $u_{\lambda,1}\ll u_{\lambda,2}$;
			\item[(iii)]  For $k=1$, and $h\leq 0$ the problem \eqref{P3} has exactly one solution.
		\end{itemize}
		
	\end{lemma}

	\begin{proof}
		We proceed in several steps.\medskip
		\newline
		{\bf Step 1:} {\it For $k>0$ small, \eqref{P3} admits a solution.}	Let $\lambda>\gamma_1$ and $\varepsilon_0>0$ be given by Lemma \ref{antimax} corresponding to $\overline{c}=c(x)$,
	 $\overline{d}=\nu_2h^-(x)$, $\overline{h}=\nu_2\widetilde{c}(x)+{k}^{-1}\nu_2h^+(x)$, and choose\\ $\displaystyle\lambda_0\in\big(\gamma_1,\min\big\{\gamma_1+\varepsilon_0,\gamma_1+{(\lambda-\gamma_1)}/{2} \big\}\big]$.
		Then, the following problem
		\begin{align*}
			-\divi(A(x)Du)+\nu_2h^-(x)u=c_{\lambda_0}u+\nu_2\widetilde{c}(x)+\frac{1}{k}\nu_2h^+(x)
		\end{align*}
		has a solution $u\ll 0$.
		Taking $\delta>0$ small enough we obtain
		\begin{align*}
			\lambda_0 s\geq (1+\lambda s)\ln(1+\lambda s) \quad 	\text{for \ all} \quad  s\in[-\delta,0].
		\end{align*}
		Defining $\widetilde{\beta}_k={k}{\lambda}^{-1}u$ for $k>0$ small enough, it follows that $\widetilde{\beta}_k\in[-\delta,0]$ and it satisfies
		\begin{align*}
			-\divi(A(x)D\widetilde{\beta}_k)&=c_{\lambda_0}\widetilde{\beta}_k+\nu_2\frac{k}{\lambda}\widetilde{c}(x)-\nu_2\frac{k}{\lambda}h^-(x)u+\frac{1}{\lambda}\nu_2h^+(x),\mbox{ and hence}\\
			-\divi(A(x)D\widetilde{\beta}_k)&+\nu_2h^-(x)\widetilde{\beta}_k=c_{\lambda_0}(x)\widetilde{\beta}_k+\nu_2\frac{k}{\lambda}\widetilde{c}(x)+\frac{1}{\lambda}\nu_2h^+(x).
		\end{align*}
		Thus, for $\beta_k$ being defined by $\beta_k={\nu_2}^{-1}\ln(1+\lambda \widetilde{\beta}_k)$, we have
		\begin{align*}
			-\divi(A(x)D\beta_k)&=-\frac{\lambda}{\nu_2}\frac{\divi(A(x)D\widetilde{\beta}_k)}{(1+\lambda\widetilde{\beta}_k)}-\frac{\lambda}{\nu_2}\left(A(x)D\widetilde{\beta}_k,D\left[{(1+\lambda\widetilde{\beta}_k)^{-1}}\right]\right)\\
			&\gneqq c_\lambda(x)\beta_k+\frac{k\widetilde{c}(x)+h^+(x)-\lambda h^-(x)\widetilde{\beta}_k}{1+\lambda \widetilde{\beta}_k}+\frac{\lambda^2}{\nu_2(1+\lambda\widetilde{\beta}_k)^2}(A(x)D\widetilde{\beta}_k,D\widetilde{\beta}_k)\\
			&\geq c_\lambda(x)\beta_k+k\widetilde{c}(x)+h^+(x)-h^-(x)+\nu_2\vartheta\frac{|D\widetilde{\beta}_k|^2}{(1+\nu_2\widetilde{\beta})^2}\\
			&\geq c_\lambda(x)\beta_k+k\widetilde{c}(x)+h(x)+(M(x)D\beta_k,D\beta_k) .
		\end{align*}
		Therefore, we conclude that
		\begin{eqnarray*}
			\left\{
			\begin{array}{rll}
				-\divi(A(x)D\beta_k)&\geq c_\lambda(x)\beta_k+k\widetilde{c}(x)+h(x)+(M(x)D\beta_k,D\beta_k) &\mbox{ in } \Omega\\
				\beta_k&=0 &\mbox{ on } \partial\Omega
			\end{array}
			\right.
		\end{eqnarray*}
		has a supersolution $\beta_k$ with $\beta_k\ll 0$ and that \eqref{P3} has at least one solution, by following the proof of Theorem \ref{teo5.2}.\medskip
		\newline
		{\bf Step 2:} {\it For $k>1$ the problem \eqref{P3} has no solution.}	First we observe that every solution to \eqref{P3} for $\lambda\in [0,\Lambda_2]$ is positive in $\Omega$.
		In fact, we observe that
		\begin{eqnarray*}
			\left\{
			\begin{array}{rll}
				-\divi(A(x)Du)&\geq \left(M(x)Du,Du\right)+Bc^+(x) 
				\geq 0&\mbox{ in } \Omega\\
				u&=0 &\mbox{ on } \partial\Omega
			\end{array}
			\right.
		\end{eqnarray*}
		and, in view of the SMP, it implies that $u>0$ in $\Omega$.
		In order to obtain a contradiction, assume that $u$ is a solution to \eqref{P3} in $\Omega$.
			Let $\varphi\in C_0^\infty(\Omega)$ such that $\varphi^2\gg 0$.
			Using $\varphi^2$ as a test function, by Theorem \ref{lowerbound} we obtain
			\begin{align*}
				\int\frac{1}{\mu_1}|D\varphi|^2&\geq 2\int(\varphi Du,D \varphi)-\mu_1\int|Du|^2\varphi^2\geq 2\int(\varphi Du,D \varphi)-\int(M(x)Du,\varphi^2Du)\\
				&\geq -\Lambda_2C_0\int c^+(x)\varphi^2-M\int c^-(x)\varphi^2-\int h^-(x)\varphi^2+\int k \tilde{c}(x)\varphi^2,
			\end{align*}
			which is a contradiction for $k > 1$ large enough.\medskip
			\newline
			{\bf Step 3:} {\it For $k=1$ \eqref{P3} has a unique solution and for $k\in(0,1)$, problem \eqref{P3} has a strict supersolution.}	By Step 1 and 2 we have
			$
				1=\sup\{k>0; \eqref{P3}\mbox{ has at least one solution}\}.
			$
			Let $k\in(0,1)$ and $\widetilde{k}\in(k,1)$ be such that ($P_{\lambda,\widetilde{k}}$) has a solution $\widetilde{\beta}$.
			Then, $\beta={k}{\widetilde{k}}^{-1}\widetilde{\beta}$ is a supersolution to \eqref{P3}. 
			Now, as in (iii) of the proof of Theorem \ref{teo5.2} we can prove that $\beta$ is a strict supersolution to \eqref{P3} and also derive the existence of the second solution $u_{\lambda,2}$ with $u_{\lambda,1}\ll u_{\lambda,2}$.
		\end{proof}
	
		\begin{lemma}\label{u_2>0}
			Under assumption \eqref{A}, assume that ($P_0$) has a solution $u_0\leq 0$ with $c^+(x)u\lneqq 0$. Then, for all $\lambda\geq 0$, problem \eqref{$P_lambda$} has at most one solution $u\leq 0$.
		\end{lemma}
		\begin{proof}
			The proof is divided in several steps.\medskip
			\newline
			{\bf Step 1:} {\it If $u$ is a subsolution to \eqref{$P_lambda$} with $u\leq 0$, then $u\ll 0$.} 	In fact, $u$ is a subsolution to ($P_0$) and by the Comparison Principle,  we have $u\leq u_0$. In addition, for $w=u_0-u$ we have
			\begin{align*}
				-\divi(A(x)Dw) &\geq -c^-(x)u_0+(M(x)Du_0,Du_0)-c_\lambda(x)u-(M(x)Du,Du)  \\
				&=(M(x)Du+Dw,Dw)-c^-(x)w-\lambda c^+(x)u,
			\end{align*}
			and hence, we get
			\begin{align*}
				\left\{
				\begin{array}{rll}
					-\divi(A(x)Dw)-(M(x)Du+Dw,Dw)-c^-(x)w 
					&\gneqq 0 &\mbox{ in } \Omega\\
					w&= 0 &\mbox{ on } \partial\Omega.
				\end{array}
				\right.
			\end{align*}
			This implies that $w \gg 0$ i.e. $u\ll u_0\leq 0$.\medskip
			\newline
			{\bf Step 2:} {\it If we have two solutions $u_1, u_2\leq 0$ to \eqref{$P_lambda$}, then such solutions are ordered as  $\tilde{u}_1\lneqq \tilde{u}_2\leq u_0$.} 	By Step 1, we have $u_1, u_2 \ll u_0$ .
			In case $u_1$ and $u_2$ are not ordered, as $u_0$ is a supersolution to \eqref{$P_lambda$}, applying
		 \cite[Theorem 2.1]{CJ} there exists a solution $u_3$ of (\ref{$P_lambda$}) with $\max\{u_1,u_2\}\leq u_3\leq u_0$. This proves Step 2 by choosing $\tilde{u}_1= u_1$ and $\tilde{u}_2 = u_3$.\medskip
			\newline
			{\bf Step 3:} {\it There exists at most one non-positive solution to ($P_{\lambda}$).}	Let us assume by contradiction that we have two ordered non-positive solutions, we can suppose $u_1\ll u_2\ll 0$. As $|u_2|\gg 0 $ the set $\{\varepsilon>0 , u_2-u_1 \leq \varepsilon|u_2|\}$ is nonempty.
			Defining
			\begin{align*}
				\tilde{\varepsilon}:=\min \{\varepsilon>0, u_2-u_1 \leq \varepsilon
				|u_2|\} \quad \text{		and \ setting} 
			\quad
				w_{\tilde{\varepsilon}}:=\frac{(1+\tilde{\varepsilon})u_2-u_1}{\tilde{\varepsilon}},
			\end{align*}
			we can use de convexity of the function $\varphi(\xi):=(M(x)\xi,\xi)$ for each $\xi\in \mathbb{R}^n$ and write $u_2={\tilde{\varepsilon}}{(1+\tilde{\varepsilon})^{-1}}w_{\tilde{\varepsilon}}+{(1+\tilde{\varepsilon})^{-1}}u_1,$ then we obtain
			\begin{align*}
					(M(x)Du_2,Du_2)&=\varphi\left(\frac{\tilde{\varepsilon}}{1+\tilde{\varepsilon}}Dw_{\tilde{\varepsilon}}+\frac{1}{1+\tilde{\varepsilon}}Du_1\right)\leq \frac{\tilde{\varepsilon}}{1+\tilde{\varepsilon}}\varphi(Dw_{\tilde{\varepsilon}})+\frac{1}{1+\tilde{\varepsilon}}\varphi\left(Du_1\right)\\
				&=\frac{1}{1+\tilde{\varepsilon}}\big[\tilde{\varepsilon}(M(x)Dw_{\tilde{\varepsilon}},Dw_{\tilde{\varepsilon}})+(M(x)Du_1,Du_1)\big], \quad \mbox{and hence}\\
				\frac{1+\tilde{\varepsilon}}{\tilde{\varepsilon}}&(M(x)Du_2,Du_2)\leq (M(x)Dw_{\tilde{\varepsilon}},Dw_{\tilde{\varepsilon}})+\frac{1}{\tilde{\varepsilon}}(M(x)Du_1,Du_1).
			\end{align*}
			\begin{align*}
		\mbox{	Thus, it yields \ }	\qquad	-\divi(A(x)Dw_{\tilde{\varepsilon}})
				&\leq \frac{1+\tilde{\varepsilon}}{\tilde{\varepsilon}}\big(c_\lambda(x)u_2+(M(x)Du_2,Du_2)+h(x)\big)\\
				&\quad-\frac{1}{\tilde{\varepsilon}}\big(c_\lambda(x)u_1+(M(x)Du_1,Du_1)+h(x)\big)\\
				&\leq c_{\lambda}(x)
				w_{\tilde{\varepsilon}}+(M(x)Dw_{\tilde{\varepsilon}},Dw_{\tilde{\varepsilon}})+h(x). \qquad \qquad \qquad \quad \qquad
			\end{align*}
			Applying again the Comparison Principle, we get $w_{\tilde{\varepsilon}}\lneqq u_2\leq 0$, which is a contradiction due to the definition of $\tilde{\varepsilon}$.
		\end{proof}
	
	Finally, we have all the necessary tools to prove Theorem \ref{teo5.3}.
		\begin{proof}[Proof of Theorem \ref{teo5.3}]
			We treat separately the cases $\lambda\leq 0$ and $\lambda>0$.\medskip
			\newline{\bf Case {(i)}:} $\lambda\leq 0$.	As in the proof of Theorem \ref{teo5.2} we can apply \cite[Theorem 1.2]{MR4030257}. Moreover, observe that $u_0$ is a subsolution to (\ref{$P_lambda$}). Hence we conclude that $u_\lambda\geq u_0$ applying the Comparison Principle. By  \cite[Proposition 4.1]{ACJT} the problem  \eqref{$P_lambda$} has at most one solution and by Lemma \ref{lbound}  the function $v=u_0+\|u_0\|_{\infty}$ is a supersolution to \eqref{$P_lambda$} when $\lambda<0$. Then, the Comparison Principle implies that $u_0+\|u_0\|_{\infty}\geq u_\lambda$ .\medskip
			\newline{\bf Case {(ii)}:} $\lambda>0$. With the aim of showing the existence of a continuum of solutions to problem (\ref{$P_lambda$}), for $\lambda\geq 0$ we introduce the auxiliary problem
			\begin{equation}\tag{$\underline{P_\lambda}$}\label{p4}
			\begin{cases}
	-\divi(A(x)Du)+u&=[c_\lambda(x)+1][u_0-(u-u_0)^-]+(M(x)Du,Du)+h(x) \quad \mbox{ in } \Omega \\
\qquad \qquad \qquad \qquad  u &= 0 \quad \mbox{ on } \partial \Omega.
\end{cases}		
			\end{equation}
Observe that $u$ is a solution to \eqref{p4} if and only if it is a fixed point of the operator $\widehat{T}_\lambda$ defined by $\widehat{T}_\lambda:C^1(\overline{\Omega})\rightarrow C^1(\overline{\Omega}):v\rightarrow u$, where $u$ is the solution to
			\begin{align*}
				-\divi(A(x)Du)+u-\big(M(x)Du,Du\big)=[c_\lambda(x)+1][u_0-(v-u_0)^-]+h(x).
			\end{align*}
			Applying the same argument to $\overline{T}_\lambda$ as the one used in the proof of Theorem \ref{teo5.2}, we see that $\widehat{T}_\lambda$ is completely continuous,
			and we split the rest of the proof into three steps.
		\medskip	\newline
			{\bf Step 1:} {\it If $u$ is a solution to (\ref{p4}) then $u\leq  u_0$ and it is a solution to (\ref{$P_lambda$}).} 
			Observe that $u_0-u-(u-u_0)^-\leq 0$. Moreover, we also have ${\lambda c^+(x)[u_0-u-(u-u_0)^-]\leq \lambda c^+(x)u_0\leq  0.}$
			Hence, we deduce that a solution $u$ of (\ref{p4}) is a subsolution to
			\begin{align}\label{5.4}
				-\divi(A(x)Du)&=-c^-(x)[u_0-(u-u_0)^-]+(M(x)Du,Du)+h(x).
			\end{align}
			Since $u_0$ is a solution to (\ref{$P_lambda$}), it implies that $u_0$ solves (\ref{5.4}). Then, applying again the Comparison Principle we get $u\leq u_0$.
		\medskip\newline
	{\bf Step 2:} {\it$u_0$ is the unique solution to ($\underline{{P}_0}$) as well as to the problem ($P_0$) and $i(I-\widehat{T}_0,u_0)=1$.}
			For $\lambda=0$, if  $u$ is a solution to (\ref{5.4}), then by Step 1, $u\leq u_0$ and $u$ solves (\ref{$P_lambda$}). From (i) we conclude that $u=u_0$.
			In order to prove that $i(I-\widehat{T}_0,u_0)=1$, we consider the operator  
			$	\widehat S_t:C^1(\overline{\Omega})\rightarrow C^1(\overline{\Omega})$ given by $ \widehat S_t(v):=t\widehat{T}_0v=u
		$,
			where $u$ is the solution to
			\begin{align*}
				-\divi(A(x)Du)+u&=(M(x)Du,Du)+th(x) +t[-c^-(x)+1][u_0-(v-u_0)^--(v-u_0+1)^-].
			\end{align*}
	By the complete continuity of $\widehat{T}$ and also by the fact that every solution $u$ to \eqref{P2} is $C^{\alpha}$ up to the boundary, there exists $R>0$ such that for all $t \in [0,1]$ and all $v \in C^1(\overline{\Omega})$, it follows that
			$
				\|S_tv\|_{C^\alpha}<R.
			$
			Then, $I-S_t$ does not vanish on $\partial B_R(0)$ and
			\begin{align*}
				\deg(I-\widehat{T}_0,B_R(0))&= \deg(I-S_1,B_R(0))=\deg(I-S_0,B_R(0))=\deg(I,B_R(0))=1.
			\end{align*}
			Therefore, $\widehat{T}_0$ has only a fixed point $u_0$, which is a solution to $(\underline{{P}_0})$. Therefore, arguing as in Step 2 of Theorem \ref{teo5.2} we conclude this step.
	\medskip
		\newline
			{\bf Step 3:} {\it Existence and behavior of the continuum.} It follows the same lines as  Step 3 of Theorem \ref{teo5.2}.
				
			For the multiplicity results in (iii), we observe that	by Step 1, we get the existence of a first solution $u_{\lambda,1}\leq u_0$. To prove that $u_0$ is a strict supersolution to \eqref{$P_lambda$}, we argue as in  the proof of Theorem \ref{teo5.2}, and by Lemma \ref{existssub} \eqref{$P_lambda$} has a strict subsolution $\alpha$ with $\alpha\leq u_0$.
			Then, by Theorem \ref{existence}, there exists $R>0$ such that $u_{\lambda,1}\in \mathcal{S}$, where \[\mathcal{S}=\{u\in C_0^1(\overline{\Omega});\alpha\ll u\ll u_0 \mbox{ in }\Omega, \|u\|_{C^1_0}<R\}.\]
			Now, fixing  $\lambda>0$ and setting $\Lambda_2=2\lambda$, we replace $h$ by $h +k\widetilde{c} $ in the problem \eqref{P3}, and then Theorem \ref{6.3} gives us an $L^\infty$ a priori bound for solutions to \eqref{P3} for every $k\in [0,1]$. This provides, by the  $C^{1,\mathcal{D}ini}$ global estimates, an a priori bound for solutions in $C^1_0(\overline{\Omega})$, i.e. $\|u\|_{C^1_0(\overline{\Omega})}<R_0$ for every  solution $u$ to \eqref{P3}, for all $k\in [0,1]$, where $R_0>R$ also depends on $\lambda$.
			Hence, by the homotopy invariance of the degree, and the fact that, for $k > 1$, \eqref{P3} has no solution we have
			\begin{eqnarray*}
				\deg(I-\widehat{T}_\lambda,B_{R_0}(0))=\deg(I-\widehat{T}_{\lambda,0},B_{R_0}(0))=\deg(I-\widehat{T}_{\lambda,k},B_{R_0}(0))=0,
			\end{eqnarray*}
			where $\widehat{T}_{\lambda,k}$ is the operator $\widehat{T}_\lambda$ in which we replace $h(x)$ by $h(x)+k\widetilde{c}$, note that $\widehat{T}_{\lambda,k}$ is clearly still completely continuous. But then, by the excision property of the degree,
			\begin{eqnarray*}
				\deg(I-\widehat{T}_\lambda,B_{R_0}\setminus \mathcal{S}(0))=\deg(I-\widehat{T}_\lambda,B_{R_0}(0))-\deg(I-\widehat{T}_\lambda,S(0))=-1
			\end{eqnarray*}
			and the existence of a second solution $u_{\lambda,2}\in B_{R_0}\setminus \mathcal{S}$ is derived.
			By Lemma \ref{u_2>0} we have $ u_{\lambda,2}>0$.
			
		 For finishing, we claim that for fixed $\lambda_1< \lambda_2$ we have $u_{\lambda_2,1}\ll u_{\lambda_1,1}$.
			In fact, note that \[c_{\lambda_1}(x)u_{\lambda_1,1}=\lambda_1 c^+(x)u_{\lambda_1,1}-c^-(x)u_{\lambda_1,1}\gneqq \lambda_2 c^+(x)u_{\lambda_1,1}-c^-(x)u_{\lambda_1,1}=c_{\lambda_2}(x)u_{
					\lambda_1,1},\]
				since $u_{\lambda_1,1}<0$. Then, $u_{\lambda_1,1}$ is a strict supersolution to $(P_{\lambda_2})$, which is not a solution and, in particular, $u_{\lambda_1,1}\ne u_{\lambda_2,1}$.
				As in the proof of  \cite[Claim 6.16]{multiplicidade}, we observe that $u_{\lambda_2,1}$ is the minimal solution to ($P_{\lambda_2}$).
				In fact, recall that $\xi=\xi_{\lambda_2}$, given by Lemma \ref{existssub}, is such that $\xi\leq u$ for every strict supersolution to ($P_{\lambda_2}$) and in particular $\xi\leq u_{\lambda_1,1}$.
				Remember also that  $u_{\lambda_2,1}$ is
				the minimal strict solution such that $u_{\lambda_2,1}\geq \xi$ in $\Omega$. Now, if there was a $x_0\in \Omega$ such
				that $u_{\lambda_2,1}(x_0)>u_{\lambda_1,1}(x_0)$, by defining $\eta:=\min \{u_{\lambda_1,1},u_{\lambda_2,1}\}$, as the minimum of the strict supersolutions to ($P_{\lambda_2}$) not less than $\xi$, we have $\xi\leq\eta$ in $\Omega$. Thus, applying again Theorem \ref{existence} we get 
				a solution $u$ of ($P_{\lambda_2}$) such that $\xi\leq u \leq\eta\lneqq u_{\lambda_2,1}$  in $\Omega$, which contradicts the minimality of $u_{\lambda_2,1}$. Ending the
			proof of Theorem \ref{teo5.3}.
		\end{proof}
		
		In what follows, we prove Theorem \ref{solucoesnegativas}, considering the alternative situation when there exists a supersolution to (\ref{$P_lambda$}) for some $\lambda_0 > 0$.
		
		\begin{proof}[Proof of Theorem \ref{solucoesnegativas}]
	For proving (i), we first observe that if \eqref{$P_lambda$} has a supersolution $\beta_\lambda\leq 0$, then $\beta_\lambda$ satisfies also $c^+(x)\beta_\lambda\lneqq 0$, otherwise, it is also a supersolution to ($P_0$),
		which contradicts the assumption (A). 
		Let us define
		\begin{equation*}
			\underline{\lambda}=\inf \{\lambda\geq 0; (\mbox{\ref{$P_lambda$}}) \mbox{ has a supersolution } \beta_\lambda\leq 0 \mbox{ with } c^+(x)\beta_\lambda\lneqq 0  \}.
		\end{equation*}
		Given $\lambda> \underline{\lambda}$, by the definition of $\underline{\lambda}$ there exists $\widetilde{\lambda}\in [\underline{\lambda},\lambda)$, such that ($P_{\widetilde{\lambda}}$) has a supersolution $\beta_{\widetilde{\lambda}}\leq 0$ with $c^+(x)\beta_{\widetilde{\lambda}}\lneqq 0$.
		Note that
		$$c_{\widetilde{\lambda}}(x)\beta_{\widetilde{\lambda}}=\widetilde{\lambda}c^+(x)\beta_{\widetilde{\lambda}}-c^-(x)\beta_{\widetilde{\lambda}}\gneqq \lambda c^+(x)\beta_{\widetilde{\lambda}}-c^-(x)\beta_{\widetilde{\lambda}}=c_{\lambda}(x)\beta_{\widetilde{\lambda}}.$$
		Then, $\beta_{\widetilde{\lambda}}$ is a supersolution to \eqref{$P_lambda$}, which is not a solution and hence, as in the proof of Theorem \ref{teo5.3}(iii), it is a strict supersolution to (\ref{$P_lambda$}).
		By Lemma \ref{existssub}, \eqref{$P_lambda$} has a strict subsolution $\alpha_\lambda\leq \beta_{\widetilde{\lambda}}$ and $\alpha_\lambda\leq u$ for all solutions $u$ to \eqref{$P_lambda$}. As in Step 2 of the proof of Theorem \ref{teo5.3}, there exists $R>0$ such that $\deg(I-\widehat{T}_\lambda,S)=1$ with
		\begin{align*}
			S=\{u\in C^1_0(\overline{\Omega})\mbox{, } \alpha\ll u\ll \beta_{\widetilde{\lambda}}\mbox{, } \|u\|_{C^1}\leq R\},
		\end{align*}
		and hence the existence of the first solution $u_{\lambda,1}\ll 0$ is derived.
		To obtain a second solution $u_{\lambda,2}$ satisfying $u_{\lambda,1}\ll u_{\lambda,2}$ and $u_{\lambda,2}>\beta_{\widetilde{\lambda}}$ we repeat the argument in the proof of the Theorem \ref{teo5.3}(iii).
		By Lemma \ref{u_2>0}  we have $u_{\lambda,2}>u_{\overline{\lambda}}$.
		Finally, arguing as the ending of the proof of Theorem \ref{teo5.3}, we prove that if $\lambda_1 < \lambda_2$ we have $u_{\lambda_1,1}\gg u_{\lambda_2,1}$.
		
		For proving that ($P_{\underline{\lambda}}$) has a non-positive solution, let $\{\lambda_n\}\subset (\underline{\lambda}, \infty)$
		be a decreasing sequence such that $\lambda_n\rightarrow \underline{\lambda}$. By the regularity result proved in  \cite[Lemma 2.1]{ACJTuni}, we know that there exists a sequence of corresponding solutions $\{u_n \}\subset H^1(\Omega)\cap W^{1,n}_{loc}(\Omega)\cap C^1(\overline{\Omega})$  with $u_n\leq u_{n+1}\leq 0$.
		As $\{u_n\}$ is increasing and bounded above, by Theorem \ref{6.3}, there exists $M>0$ such that  $\|u_n\|_{L^\infty}<M$ for all $n\in \mathbb{N}$, and hence by the $C^{1,\mathcal{D}ini}$ global estimates, we get $\|u_n\|_{C^{1,\mathcal{D}ini}}\leq C$.
		Then, up to a subsequence, $u_n\rightarrow u$ in $C^{1}_0(\Omega)$, which allows us to conclude that $u$ is a solution to  ($P_{\underline{\lambda}}$) with $u\leq 0$.
		
		Now we prove the uniqueness of the non-positive solution to  ($P_{\underline{\lambda}}$).
		Let us assume by contradiction that we have two distinct solutions, $u_1$ and $u_2$ of  ($P_{\underline{\lambda}}$), then as in the Step 3 of the proof of Theorem \ref{teo5.3}, we prove that $\beta=(u_1+u_2)/2$ is a strict supersolution to  ($P_{\underline{\lambda}}$). Let us consider the strict subsolution $\alpha\ll \beta$ of  ($P_{\underline{\lambda}}$) given by Lemma \ref{existssub}, and define the set
	$
			\overline{\mathcal{S}}=\{u\in C_0^1(\overline{\Omega})\mbox{; } \alpha\ll u\ll \beta,\|u\|_{C^1_0}<R\}
	$
		for some $R>C>0$. Again, by the $C^{1,\mathcal{D}ini}$ estimates, we have that
		\begin{align}\label{underlambda}
			\|u\|_{C^{1,\mathcal{D}ini}}\leq C \mbox{ for all } u \mbox{ solution  of } (P_\lambda)\mbox{, } \lambda\in[\underline{\lambda}-1, \underline{\lambda}]
		\end{align}		
		such that $deg(I-\widehat{T}_{\underline{\lambda}},\mathcal{\overline{\mathcal{S}}})=1$. Now we prove the existence of $\varepsilon>0$ such that
		\begin{align}\label{existsvarepsilon}
			deg(I-\widehat{T}_\lambda,\overline{\lambda})=1\mbox{, for all } \lambda\in[\underline{\lambda} -\varepsilon,\underline{\lambda}].
		\end{align}
	
		We argue as the end of the proof of Theorem \ref{teo5.2}, in order to verify that there exists some $\varepsilon\in(0,1)$ such that there is no fixed points of $T_\lambda$ on the boundary of $\overline{\mathcal{S}}$ for all $\lambda$ in the preceding interval. 
		Hence for obtaining \eqref{existsvarepsilon} it is sufficient to apply the homotopy invariance for $\lambda \in [\underline{\lambda}-\varepsilon,\underline{\lambda}]$. Next, with \eqref{existsvarepsilon} in hand, we repeat exactly the same argument done above to obtain the existence of a second solution $u_{\lambda,2}$ to \eqref{$P_lambda$}, for all $\lambda\in [\underline{\lambda}-\varepsilon,\underline{\lambda}]$. But this, contradicts the definition of $\underline{\lambda}$ completing the proof of (ii).
		 By the definition of $\underline{\lambda}$ and since $\beta$ is a strict supersolution to \eqref{$P_lambda$} we infer that (iii) holds.
	
	Finally, in order to describe the behavior of the solutions for $\lambda\rightarrow 0^-$, observe that in Lemma \ref{lbound} we have proved that  $\|u_\lambda\|_\infty< 2\|u_{\widehat{\lambda}}\|_\infty$ for all $\lambda\leq \widehat{\lambda}<0$.	In particular, if $C_0:= \displaystyle\limsup_{\lambda\rightarrow  0^-}\|u_\lambda\|_\infty<\infty$, there exists a sequence $\widehat{\lambda}_n\rightarrow 0^-$ such that ${C_0=\displaystyle\limsup_{n \rightarrow \infty} \|u_{\widehat{\lambda}_n}\|_\infty <\infty}$. Then, for every sequence $\lambda_n\rightarrow 0^-$ we deduce by the above inequality that $\displaystyle\limsup_{n \rightarrow \infty} \|u_{\lambda_n}\|_{\infty}\leq2C_0<\infty$. Therefore, we have either $\displaystyle\lim_{\lambda\rightarrow 0^-} \|u_\lambda\|_\infty= \infty$ or $\displaystyle\limsup_{\lambda\rightarrow 0^-} \|u_\lambda\|_\infty< \infty$.
		By assumption, we know that ($P_0$) does not have a solution $u_0\leq 0$, then the first case holds, finishing the proof.
		\end{proof}
		
		\begin{proof}[Proof of Corollary \ref{coro}]
		First observe that ($P_{\gamma_1,k}$) has no solution.
		In fact, if $u$ is a solution to \eqref{P3}, using $\varphi_1>0$, the first eigenfunction of \eqref{eig1}, as a test function in \eqref{P3}, we have
		\begin{align*}
			\int c_{\gamma_1}(x)u\varphi_1 =\int A(x)DuD \varphi_1=\int c_\lambda(x)u&\varphi_1 +\int \varphi_1(M(x)Du,Du)+\int (h(x)+k\widetilde{c}(x))\varphi_1 \\
			\mbox{and} \quad
			(\gamma_1-\lambda)\int c^+(x)u\varphi_1 \leq -\int|h(x)|&\varphi_1<0, \quad \mbox{which is a contradiction for \ } \lambda=\gamma_1. \qquad \quad
		\end{align*}
		Hence, for all $\lambda>0$, problem \eqref{$P_lambda$} has no solution with $c^+(x)u \equiv 0$, otherwise $u$ would be a solution to \eqref{$P_lambda$} for every $\lambda\in \mathbb{R}$, which contradicts the nonexistence of a solution for $\lambda=\gamma_1$.
		By Step 3 of the proof of Lemma \ref{c2}, there exists $\widetilde{k}>0$ such that for all
		$k\in (0,\widetilde{k}]$ problem \eqref{P3} has a strict supersolution $\beta_0$ with $\beta_0\ll 0$. The existence of $\lambda_2>\gamma_1$ as in (iii) can then be deduced from Theorem \ref{solucoesnegativas}.
		By   \cite[Theorem 1.1]{ACJT}, decreasing $\widetilde{k}$ if necessary, we know that for all $k\in(0,\widetilde{k}]$ problem ($P_{0,k}$) has a solution $u_0\gg 0$. Therefore, the existence of $\lambda_1$ as in (i) can be deduced from Theorem \ref{teo5.2}.
		\end{proof}
		
		Before proving Theorem \ref{hequiv0}, we observe that special cases of Theorems \ref{teo5.2} and \ref{teo5.3} are given when $h(x) \gneqq 0$ and $h(x) \lneqq 0$, respectively.
			Indeed, if $h\gneqq 0$ holds, then $u_0$ is a supersolution to
			\begin{align*}
				\left\{
				\begin{array}{rll}
					-\divi(A(x)Du_0)&\geq c_\lambda(x)u_0+(M(x)Du_0,Du_0)+h(x)\gneqq 0&\mbox{ in } \Omega\\
					u_0&=0 &\mbox{ on } \partial\Omega
				\end{array}
				\right.
			\end{align*}
			Then, applying the Strong Maximum Principle we obtain $u_0>0$ in $\Omega$. Furthermore, by the Hopf Lemma we conclude that $u_0\gg 0$ in $\Omega$.
			On the other hand , if $h \lneqq 0$, then $u_0$ is a subsolution to
			\begin{align*}
				-\divi(A(x)Du_0)\leq c_\lambda(x)u_0+(M(x)Du_0,Du_0)+h(x) \lneqq (M(x)Du_0,Du_0)
			\end{align*}
			and so $v_0={\nu_2}^{-1}(e^{\nu_2u_0}-1)$  is a subsolution to
			\begin{align*}
				-\divi(A(x)Dv_0)&\leq [1+\nu_2 v][-\divi(A(x)Du_0)-\mu_2|Du_0|^2]\\
				&\lneqq [1+ \nu_2v][(M(x)Du_0,Du_0)-\mu_2|Du_0|^2]\lneqq 0 \mbox{ in }\Omega.
			\end{align*}
			Again, by the SMP and the Hopf Lemma we get $v_0\ll 0$ and therefore $u_0\ll 0$.
		
		\begin{proof}[Proof of Theorem \ref{hequiv0}]
		For proving (i), firstly we note that for all $\lambda\in \mathbb{R}$, $u\equiv 0$ is a solution to (\ref{h=0}). In order to prove that for all $\lambda\in(0,\gamma_1)$  problem (\ref{h=0}) has a second solution $u_{\lambda,2}\gneqq 0$, we claim that  problem (\ref{h=0}) has a supersolution $\beta\gg 0$. In fact, taking $\lambda< \gamma_1$ and $\varepsilon>0$ such that
		$
				\lambda{\nu_2}^{-1}{(1+\nu_2v)\ln(1+\nu_2v)}\leq \gamma_1v
		$ for all $v\in[0,\varepsilon]$, we consider the function $\widetilde{\beta}=\varepsilon\varphi_1,$ where $\varphi_1$ denotes the first eigenfunction of \eqref{eig1}
			with $\|\varphi_1\|_{L^{\infty}}=1$ and
			\begin{align*}
				\left\{
				\begin{array}{rll}
					-\divi(A(x)D\widetilde{\beta})&=c_\gamma(x)\widetilde{\beta}
					\gneqq c_\lambda(x) \displaystyle\frac{(1+\nu_2\widetilde{\beta})\ln(1+\nu_2\widetilde{\beta})}{\nu_2}, &\mbox{ in } \Omega\\
					\widetilde{\beta}&=0 &\mbox{ on } \partial\Omega.
				\end{array}
				\right.
			\end{align*}
			Hence, for $\beta$ being defined by $\beta:={\nu_2}^{-1}{\ln(1+\nu_2\widetilde{\beta})}$, we have
			\begin{align*}
				-\divi(A(x)D\beta)&=-\frac{\divi(A(x)D\widetilde{\beta})}{(1+\nu_2\widetilde{\beta})}-\left(A(x)D\widetilde{\beta},D\left[(1+\nu_2\widetilde{\beta})^{-1}\right]\right)\\
				&\gneqq c_\lambda(x)\beta+\frac{\nu_2}{(1+\nu_2\widetilde{\beta})^2}(A(x)D\widetilde{\beta},D\widetilde{\beta})\geq c_\lambda(x)\beta+\nu_2\vartheta\frac{|D\widetilde{\beta}|^2}{(1+\nu_2\widetilde{\beta})^2}\\
				&= c_\lambda(x)\beta+\mu_2|D\beta|^2\geq c_\lambda(x)\beta+(M(x)D\beta,D\beta)
			\end{align*}
			\begin{align*}
			\mbox{and thus,} \quad	\qquad \qquad \left\{
				\begin{array}{rll}
					-\divi(A(x)D\beta)& \gneqq c_\lambda(x)\beta+(M(x)D\beta,D\beta) &\mbox{ in } \Omega\\
					\beta&=0 &\mbox{ on } \partial\Omega.
				\end{array}
				\right. \qquad \qquad \qquad
			\end{align*}
			By the Comparison Principle, we have that $\beta\geq 0$ is a strict supersolution to (\ref{h=0}).
			Since we know that every solution $u$ of problem (\ref{h=0}) satisfies $u\geq 0$, by Lemma \ref{existssub} problem (\ref{h=0}) has a strict subsolution $\alpha\lneqq 0$.
		Therefore, we conclude that problem (\ref{h=0}) has at least two solutions, by following the proof of Theorem \ref{teo5.2} with the solution $u_{\lambda,1}\equiv 0$.
			
			For proving (ii), suppose by contradiction that $u\not\equiv 0$ is another solution to problem (\ref{h=0}) and use $\varphi_1>0 $ the first eigenfunction of \eqref{eig1} as a test function in \eqref{h=0}. Then, 
			\begin{align*}
				\int c_{\gamma_1}(x)u\varphi_1&=\int A(x)DuD \varphi_1 =\int c_\lambda(x)u\varphi_1 +\int (M(x)Du,Du)\varphi_1 \\
				(\gamma_1-\lambda)\int c^+(x)u\varphi_1 &=\int (M(x)Du,Du)\varphi_1\geq\mu_1\int|Du|^2\varphi_1>0,
			\end{align*}
			which provides a contradiction for $\lambda=\gamma_1$.
			
			In case $\lambda>\gamma_1$, for showing that  problem \eqref{h=0} has a second solution $u_{\lambda,2}\ll 0$, let  $\lambda_0\in(\gamma_1,\lambda]$ be such that by Lemma \ref{antimax}, the problem
			\begin{align*}
				-\divi(A(x)Du)=c_{\lambda_0}(x)u+1,
			\end{align*}
			has a solution $u\ll 0$. Then, for $\varepsilon>0$ small enough, the function $\beta_0=\varepsilon u$ satisfies
			\begin{align*}
				-\divi(A(x)D\beta_0)&=c_{\lambda_0}(x)\varepsilon u+\varepsilon\geq c_{\lambda_0}(x)\beta_0+\varepsilon^2\mu_2|Du|^2\geq c_{\lambda_0}(x)\beta_0+(M(x)D\beta_0,D\beta_0)
			\end{align*}
			and the problem (\ref{h=0}) has a supersolution $\beta_0$ with $\beta_0\leq 0$ and $c^+(x)\beta_0\lneqq 0$. Therefore, (iii) follows by Theorem \ref{solucoesnegativas} with $u_{\lambda,2}\equiv 0$.
		
			With the aim of showing the existence of a continuum of solution to problem \eqref{h=0}, we define the operator  
			$
				T_{\lambda}:=
				\left\{
				\begin{array}{rll}
					\overline{T}_\lambda, & \mbox{ if } \lambda\leq \gamma_1,\\
					\widehat{T}_\lambda,& \mbox{ if } \lambda> \gamma_1.
				\end{array}
				\right.
			$
			where  $\overline{T}_\lambda$ for $\lambda\leq \gamma_1$ is defined in (ii) of the proof of Theorem \ref{teo5.2} and the operator $\widehat{T}_\lambda$ for $\lambda\geq \gamma_1$ is defined in (ii) of the proof of Theorem \ref{teo5.3} in both cases with $h\equiv 0$.
		Observe that the case $\lambda\in (-\infty,\gamma_1]$ can be proved as in the proof of Theorem \ref{teo5.2}(ii). Note that, if $u$ is a solution to (\ref{P2}), then $u\geq  u_{\gamma_1}$ and hence it is a solution to (\ref{h=0}). On the other hand, the case $\lambda\in [\gamma_1,+\infty)$ follows the same lines of the proof of Theorem \ref{teo5.3}(ii). If $u$ is a solution to (\ref{p4}), then $u\leq  u_{\gamma_1}$ and hence it is a solution to (\ref{h=0}). Furthermore, since $u_{\gamma_1}\equiv 0$ is the unique solution to the problem (\ref{h=0}) for $\lambda=\gamma_1$, then $i(I-T_{\gamma_1},u_{\gamma_1})=1$.
		
			Applying Theorem \ref{continuum} with\ $\gamma_1>0$, \ we obtain a \ continuum ${\mathcal{C}=\mathcal{C}^+\cup \mathcal{C}^-\subset \overline{\Sigma}}$ \ such that \
			${	\mathcal{C}^+=\mathcal{C}\cap ([{\gamma_1},+\infty)\times C(\overline{\Omega}))\mbox{ and } \mathcal{C}^-=\mathcal{C}\cap ((-\infty,{\gamma_1}]\times C(\overline{\Omega})) \ \mbox{are unbounded in \ } \mathbb{R}^{\pm}\times C(\overline{\Omega}).
		 }$
			By Step 1, we get that if $u \in \mathcal{C}^-$, then $u\geq u_{\gamma_1}$ and it is a solution to \eqref{h=0}. Thus, from (iv) we infer that the projection of  $\mathcal{C}^-$ on $\lambda$-axis is $(0,\gamma_1]$, a bounded interval, 
			and then we deduce that the projection of $\mathcal{C}^+$ on $\lambda$-axis is $[\gamma_1,+\infty)$.
			Hence,
		$
				\mbox{Proj}_{\mathbb{R}}\mathcal{C}=\mbox{Proj}_{\mathbb{R}}\mathcal{C}^-\cup\mbox{Proj}_{\mathbb{R}}\mathcal{C}^+=(0,+\infty).
		$
			Finally, by Theorem \ref{6.3} for any $0<\Lambda_1<\Lambda_2<\gamma_1$ there is a priori bound for the solutions to problem \eqref{h=0}, for all $\lambda\in[\Lambda_1,\Lambda_2]$.
			Then, we have also a $C^{\alpha}$ a priori bound for these solutions, i.e. the projection of $\mathcal{C}\cap ([\Lambda_1,\Lambda_2]\times C(\overline{\Omega}))$ on $C(\overline{\Omega})$ is bounded.
			Since the component $\mathcal{C}^-$ is unbounded in $\mathbb{R}^-\times C(\overline{\Omega})$, its projection on the $C(\overline{\Omega})$ axis must be unbounded.
			Therefore,  $\mathcal{C}$ must emanate from infinity on the right of axis $\lambda=0$.
		\end{proof}

\noindent{\bf Acknowledgment:} The author Mayra Soares would like to thank the financial support received by the postdoctoral fellowship from DGAPA-Unam. 

\bigskip
\noindent\textsc{Fiorella Rend\'on}\\
Departamento de Ciencias,\\
Universidad Continental\\
Arequipa, Per\'u\\
\noindent\texttt{fiorellareg@gmail.com}
\bigskip
 
\noindent\textsc{Mayra Soares}\\
Departamento de Matem\'atica,\\ Universidade de Bras\'ilia - UnB\\
Instituto Central de Ci\^encias, Campus Darci Ribeiro,\\
70910-900, Asa Norte, Bras\'ilia, Distrito Federal, Brasil\\
\noindent\texttt{mayra.soares@unb.br}

\end{document}